\documentclass[11pt]{article}

\usepackage[T1]{fontenc}
\usepackage[utf8]{inputenc}
\usepackage[margin=1in]{geometry}
\usepackage{lmodern}
\usepackage{amsmath,amssymb,amsthm,mathtools,bm}
\usepackage{booktabs,array,tabularx,multirow}
\usepackage{enumitem}
\usepackage{microtype}
\usepackage{xcolor}
\usepackage{graphicx}
\usepackage{caption}
\usepackage{subcaption}
\usepackage{algorithm}
\usepackage[noend]{algpseudocode}
\usepackage{tikz}
\usetikzlibrary{arrows.meta,positioning,fit,calc,shapes.geometric}
\usepackage[numbers,sort&compress]{natbib}
\usepackage[hidelinks]{hyperref}
\usepackage[nameinlink,capitalize,noabbrev]{cleveref}
\usepackage{xurl}

\newcommand{\R}{\mathbb R}
\newcommand{\E}{\mathbb E}
\newcommand{\Pp}{\mathbb P}
\newcommand{\norm}[1]{\left\lVert#1\right\rVert}
\newcommand{\ip}[2]{\left\langle#1,#2\right\rangle}

\newcommand{\St}{\operatorname{St}}
\newcommand{\TV}{\operatorname{TV}}
\newcommand{\KL}{\operatorname{KL}}
\newcommand{\Gap}{\operatorname{Gap}}
\newcommand{\logp}{\log_+}
\newcommand{\Tall}{\mathsf T_{\rm all}}
\newcommand{\Tloc}{\mathsf T_{\rm all}^{\rm loc}}

\newcommand{\cF}{\mathcal F}

\newcommand{\argmin}{\operatorname*{arg\,min}}
\newcommand{\argmax}{\operatorname*{arg\,max}}

\newcommand{\proxouter}{\textsc{Localized-Prox}}

\newtheorem{theorem}{Theorem}[section]
\newtheorem{lemma}[theorem]{Lemma}

\newtheorem{corollary}[theorem]{Corollary}

\theoremstyle{definition}
\newtheorem{definition}[theorem]{Definition}

\setlist[itemize]{leftmargin=*,topsep=3pt,itemsep=2pt}
\setlist[enumerate]{leftmargin=*,topsep=3pt,itemsep=2pt}
\renewcommand{\arraystretch}{1.12}
\newcommand{\authornames}{Haihan Zhang, Wendao Wu, Chenheng Zhang, Yanyi Li, Chunyuan Zheng, Cong Fang, Haoxuan Li, Zhouchen Lin}
\newcommand{\affiliation}{Peking University}
\newcommand{\emailA}{colombo.2026@outlook.com}

\newcommand{\mail}[1]{\href{mailto:#1}{\texttt{#1}}}
\newif\ifanonymous
\anonymousfalse
\ifanonymous
\author{Anonymous Authors}
\hypersetup{pdfauthor={Anonymous Authors}}
\else
\author{Wendao Wu$^{1,*}$\quad Haihan Zhang$^{1,*}$\quad Chenheng Zhang$^{1,*}$\\[0.20em]
Yanyi Li$^1$\quad Chunyuan Zheng$^1$\\[0.20em]
Cong Fang$^{1,\dagger}$\quad Haoxuan Li$^{1,\dagger}$\quad Zhouchen Lin$^{1,\dagger}$\\[0.60em]
\small $^1$\affiliation\\[0.35em]
\footnotesize \mail{\emailA} \\[0.35em]
\small $^*$Equal contribution.\quad $^\dagger$Corresponding authors.}
\hypersetup{pdfauthor={\authornames}}
\fi
\newcommand{\ResearchAgentSystem}{our laboratory's internal auto-research system}

\title{\textbf{Sharp Fresh-Gradient Complexity of Nonconvex--Strongly-Concave Minimax Optimization}}
\date{}
\hypersetup{pdftitle={Sharp Fresh-Gradient Complexity of Nonconvex-Strongly-Concave Minimax Optimization}}
\begin{document}
\maketitle

\begin{abstract}
We establish matching upper and lower bounds, up to logarithmic factors, for finding stationary points in smooth nonconvex--strongly-concave minimax optimization with fresh stochastic gradients. The bounds establish the necessity of the linear condition-number dependence in the global stochastic cost and identify additional costs of statistical accuracy and dual initialization. Let $\Phi(x)=\max_y f(x,y)$, where $f$ is jointly $L$-smooth and $\mu$-strongly concave in $y$ on unconstrained Euclidean domains, and write $\kappa=L/\mu$. Given an initial primal gap at most $\Delta$ and dual residual $\|\nabla_y f(x_0,y_0)\|\le G$, the fixed-budget complexity of obtaining $\|\nabla\Phi(\widehat x)\|\le\epsilon$ with probability at least $2/3$ is
\[
\widetilde\Theta\!\left(
\frac{\sqrt\kappa L\Delta}{\epsilon^2}
+\frac{\kappa L\Delta\sigma^2}{\epsilon^4}
+\frac{\kappa^2\sigma^2}{\epsilon^2}
+\sqrt\kappa\log_+\frac{G}{\epsilon\sqrt\kappa}
\right),
\]
where each query returns a fresh unbiased joint gradient with conditional variance at most $\sigma^2$, and $\log_+u=\log\max\{1,u\}$. The characterization holds when $\kappa$ and $L\Delta/\epsilon^2$ exceed universal constants, against arbitrary adaptive randomized algorithms and uniformly over finite dimensions. The suppressed logarithms do not depend on $G$. Dual initialization enters only logarithmically and additively, yet removing its control eliminates every finite dimension-free bound, even with exact gradients. The upper bound separates coarse primal progress from one final statistical refinement; the lower bounds show why both costs are necessary while preserving strong concavity under unrestricted queries.
\end{abstract}

\begin{samepage}
\begingroup\frenchspacing\looseness=-1
\noindent\textbf{AI Usage.}\ 
Nearly the entire research pipeline for this paper was carried out by Colombo, 
\ResearchAgentSystem{}, powered by GPT-5.6 Sol. The system also conducted a
Lean-backed article audit of the resulting manuscript. The authors subsequently reviewed and approved the
mathematical claims, presentation, and formal artifacts, and take
responsibility for the final manuscript. The complete Lean audit report and the system's technical report will be made public at a later date.\par
\endgroup
\end{samepage}

\paragraph{Publication Intent and Recognition of Independent Work.} 
This preprint is released as a research output of our study of
automated scientific research, for public scrutiny and discussion.
As we do not intend to submit this manuscript to a journal or
conference, it has not undergone the full editorial refinement
and internal review we would normally require for submission.
The authors remain responsible for its contents and for correcting
any errors or inaccurate attribution.

We recognize that releasing these results may affect researchers
who have independently been pursuing the same questions, including
work that is not yet public.
If this applies to your work, we invite you to contact us at
\texttt{colombo.2026@outlook.com}.
A brief description of the overlap, together with any dated
materials you are comfortable sharing, would help us understand
its scope and chronology.
Such materials might include draft excerpts, presentation slides,
or revision records; a complete manuscript or unpublished proof
is not required to begin the conversation.

We will review these cases carefully and update citations,
acknowledgments, and accounts of independent or concurrent work
as appropriate, including revising priority claims where warranted.
Nonpublic materials will be treated confidentially and used only
to assess the overlap unless their authors authorize otherwise.
We will obtain explicit permission from the authors before
publicly discussing such materials or providing them to an
automated research system.

\section{Introduction}

We study the first-order complexity of
\begin{equation}\label{eq:intro-problem}
    \min_{x\in\R^{d_x}}\max_{y\in\R^{d_y}} f(x,y),
\end{equation}
where $f$ is jointly $L$-smooth and $\mu$-strongly concave in $y$, with no convexity assumption on $x$. Strong concavity gives a unique inner maximizer $y^\star(x)$ and a differentiable primal function $\Phi(x)=\max_y f(x,y)$, with $\nabla\Phi(x)=\nabla_x f(x,y^\star(x))$. Our goal is to find $\widehat x$ such that $\|\nabla\Phi(\widehat x)\|\le\epsilon$. Each query supplies one fresh unbiased estimate of the joint gradient, with conditional variance at most $\sigma^2$. Thus, the algorithm must use noisy gradients at primal--dual pairs to make a gradient defined through exact inner maximization small.

Let $\kappa=L/\mu$ and let $\Delta$ bound the initial primal gap $\Phi(x_0)-\inf_x\Phi(x)$. Deterministic proximal acceleration achieves the polynomial rate $\widetilde O(\sqrt\kappa L\Delta/\epsilon^2)$ \citep{LinJinJordan2020}, with matching lower bounds for linear-span and zero-respecting methods \citep{Zhang2021,Li2021}. In the stochastic setting, the known bounds leave a gap in the dependence on $\kappa$. The zero-respecting lower bound of \citet{Li2021} contains the term $\Omega(\kappa^{1/3}L\Delta\sigma^2/\epsilon^4)$ under an expected stationarity criterion. On the upper-bound side, SAPD+ has a leading stochastic term $O(\kappa L\mathcal G_0\sigma^2/\epsilon^4)$, where $\mathcal G_0$ is its initial primal--dual gap, together with a conversion to true primal-gradient stationarity \citep{SAPDPlus2022}. The linear-$\kappa$ upper benchmark is therefore already available. The unresolved issue is whether this dependence is necessary for general randomized algorithms and how much complexity is hidden by the initialization scale.

We answer these questions by characterizing the complexity in terms of separate primal and dual initialization bounds. The resulting rate proves the necessity of the linear-$\kappa$ global stochastic cost, identifies an additional statistical cost that can dominate it, and quantifies the cost of dual initialization. Separating these contributions is essential: neither a leading $\epsilon^{-4}$ rate nor a single initial primal--dual gap describes all three effects.

\paragraph{Why dual initialization matters.}
A primal-gap bound alone gives no finite dimension-free query bound in the nontrivial parameter regime, even for exact gradients. In our construction, a hidden sign along a dual chain determines which primal outputs are stationary. A source at the far endpoint affects the primal gradient by an amount that decays exponentially with the chain length, on the scale $\sqrt\kappa$. Increasing the source compensates for this decay, keeping the hidden sign relevant to stationarity beyond any prescribed query horizon. The construction prevents dense or very large queries from revealing the sign prematurely. Thus, strong concavity ensures a unique inner solution without bounding the information needed to resolve its effect on the primal gradient.

We control the missing scale by requiring $\|\nabla_y f(x_0,y_0)\|\le G$. A finite $G$ limits this compensation for decay, leading to a logarithmic localization cost. Strong concavity and smoothness also allow a dual-distance bound to parameterize initialization; \Cref{sec:model} gives the precise relation.

\paragraph{The complexity and its consequences.}
Let $\Tloc$ be the smallest fixed query budget that guarantees $\|\nabla\Phi(\widehat x)\|\le\epsilon$ with probability at least $2/3$ on every instance with primal gap at most $\Delta$ and initial dual residual at most $G$. We prove
\begin{equation}\label{eq:intro-rate}
    \Tloc=\widetilde\Theta\!\left(
       \frac{\sqrt\kappa L\Delta}{\epsilon^2}
       +\frac{\kappa L\Delta\sigma^2}{\epsilon^4}
       +\frac{\kappa^2\sigma^2}{\epsilon^2}
       +\sqrt\kappa\log_+\frac{G}{\epsilon\sqrt\kappa}
    \right)
\end{equation}
for $\kappa\ge\kappa_0$ and $L\Delta\ge C_0\epsilon^2$, with universal constants $\kappa_0,C_0$. Here $\log_+u=\log\max\{1,u\}$, and the suppressed factors are logarithmic in $\kappa$, $L\Delta/\epsilon^2$, and $\kappa^2\sigma^2/\epsilon^2$, with the dependence on $G$ explicit. The lower bounds apply to arbitrary adaptive randomized algorithms; hard-instance dimensions are finite but may grow with the proposed budget.

The global stochastic term closes the polynomial gap from $\kappa^{1/3}$ to $\kappa$ while removing the zero-respecting restriction. Its dependence on $\Delta$ reflects the cost of making progress across the primal landscape. The term $\kappa^2\sigma^2/\epsilon^2$ captures a different obstruction: even when little global progress is needed, the algorithm must distinguish problems with different acceptable primal outputs. We construct two such problems whose oracle means differ by only $O(\epsilon/\kappa)$ at every query. Noise then makes their distinction require $\Omega(\kappa^2\sigma^2/\epsilon^2)$ calls. This cost exceeds the global stochastic term when $L\Delta/\epsilon^2<\kappa$, so it cannot be discarded in a characterization uniform in both accuracy and condition number.

Dual initialization has a different role again. Although its control is necessary for a finite uniform bound, its cost is additive and logarithmic: increasing $G$ does not multiply either stochastic term. Setting $\sigma=0$ recovers the known deterministic polynomial rate with explicit initialization dependence and an unrestricted randomized lower bound. Together, these conclusions explain what the full characterization adds to a leading-term comparison.

\Cref{tab:comparison} summarizes the principal benchmarks and their comparison conditions. The table distinguishes leading stochastic rates from our uniform characterization; \Cref{sec:related} discusses the methods behind these rates and the role of alternative sampling assumptions.

\begin{table}[t]
\centering
\caption{First-order complexity benchmarks for NC--SC minimax optimization.Rates suppress fixed smoothness, gap, and prior initialization scales, with fixed positive variance in stochastic rows. PG denotes $\|\nabla\Phi\|\le\epsilon$; expected PG is an expected-norm guarantee. Write $D_G=\sqrt\kappa\log_+(G/(\epsilon\sqrt\kappa))$.}
\label{tab:comparison}
\small
\setlength{\tabcolsep}{4pt}
\renewcommand{\arraystretch}{1.18}
\begin{tabularx}{\textwidth}{@{}>{\raggedright\arraybackslash}p{0.22\textwidth}>{\centering\arraybackslash}p{0.19\textwidth}>{\centering\arraybackslash}p{0.19\textwidth}>{\raggedright\arraybackslash}X@{}}
\toprule
Work & Upper bound & Lower bound & Scope and initialization \\
\midrule
\multicolumn{4}{@{}l}{\emph{Exact-gradient queries}}\\
\citet{LinJinJordan2020,Zhang2021}
& $\widetilde O(\sqrt\kappa\epsilon^{-2})$
& $\Omega(\sqrt\kappa\epsilon^{-2})$
& PG; dual initialization enters the upper bound; linear-span lower bound.\\
\textbf{This Paper}\newline (Corollary~\ref{cor:paper-deterministic})
& \multicolumn{2}{c}{\boldmath$\widetilde\Theta(\sqrt\kappa\epsilon^{-2}+D_G)$}
& PG with success $2/3$; residual bound $G$; unrestricted randomized methods.\\
\midrule
\multicolumn{4}{@{}l}{\emph{Fresh stochastic-gradient calls: bounded variance}}\\
Stochastic GDA\newline\citep{LinJinJordanGDA2020}
& $O(\kappa^3\epsilon^{-4})$ & ---
& Expected PG; bounded dual domain; minibatch gradients.\\
Smoothed GDA\newline\citep{YangOrvieto2022}
& $\widetilde O(\kappa^2\epsilon^{-4})$ & ---
& Expected PG after conversion; dual PL; $\epsilon^2\lesssim\kappa^{-3}$ up to logarithms.\\
SAPD+\newline\citep{SAPDPlus2022}
& $O(\kappa\epsilon^{-4})$ & ---
& Leading term at fixed $\kappa$; primal--dual gap; PG conversion available.\\
\citet{Li2021}
& --- & $\Omega\!\left(\begin{gathered}\sqrt\kappa\epsilon^{-2}\\+\kappa^{1/3}\epsilon^{-4}\end{gathered}\right)$
& Expected stationarity; zero-respecting algorithms.\\
\textbf{This Paper}\newline (\Cref{thm:paper-main})
& \multicolumn{2}{c}{\boldmath$\widetilde\Theta\!\left(\begin{gathered}\sqrt\kappa\epsilon^{-2}+\kappa\epsilon^{-4}\\+\kappa^2\epsilon^{-2}+D_G\end{gathered}\right)$}
& PG with success $2/3$ and a fixed budget; primal gap and dual residual; unrestricted methods.\\
\bottomrule
\end{tabularx}
\end{table}

\paragraph{Proof ideas.}
The lower bound must preserve both the information barrier and strong concavity. A generic rotation of a hard instance can mix primal and dual directions and destroy the latter, so we hide the two sets of coordinates separately. Dense queries create another difficulty: their small projections onto hidden directions can leak information. We make the informative links exactly flat below a threshold and modify the objective outside a protected region so that arbitrarily large queries cannot bypass the chain. For the global stochastic term, the oracle reveals the next stage only with probability of order $\epsilon^2/(\kappa\sigma^2)$ when noise dominates. Traversing $\Theta(L\Delta/\epsilon^2)$ stages therefore costs $\Omega(\kappa L\Delta\sigma^2/\epsilon^4)$ queries. Separate local-testing and dual-chain instances supply the other obstructions.

The upper bound separates the accuracy needed for descent from the accuracy needed at the output. Coarse proximal solves suffice to make primal progress; only one selected subproblem needs the finer accuracy that controls the true primal gradient. Paying this refinement once produces the local statistical term without multiplying it by the number of outer stages. To keep initialization additive as well, the inner solver controls gradient-estimation bias relative to its current optimization error. This lets early iterations tolerate coarse estimates while the error contracts; the statistical work is concentrated at the final accuracy, rather than being paid at every initialization-dependent contraction step. \Cref{sec:upper-overview,sec:lower-overview} develop these mechanisms, with formal statements in \Cref{sec:model}.
\section{Related Works}\label{sec:related}

\subsection{Deterministic minimax algorithms and lower bounds}
Two-timescale gradient descent ascent uses faster dual updates to track the moving inner maximizer \citep{LinJinJordanGDA2020}. Proximal acceleration instead reduces NC--SC optimization to strongly-convex--strongly-concave (SCSC) subproblems \citep{LinJinJordan2020}. The lower-bound constructions of \citet{Zhang2021} and \citet{Li2021} explain the difficulty of propagating information through the primal--dual coupling for restricted first-order methods. The former also treats averaged-smooth finite sums. Our lower-bound analysis must preserve this obstruction when queries need not lie in the span of previous gradients.

Proximal reductions also have an earlier stochastic foundation. \citet{RafiqueEtAl2022} develop proximally guided stochastic subgradient and variance-reduced methods for weakly-convex--concave objectives, including nonsmooth problems. Their nearly stationary outer solutions explain the role of proximal subproblems beyond the smooth NC--SC setting. Our analysis uses this reduction principle but requires a true primal-gradient output and explicit fresh-oracle accounting.

\subsection{Stochastic NC--SC optimization and sampling access}
Stochastic gradient descent ascent uses unbiased gradient estimates without requiring common-sample differences \citep{LinJinJordanGDA2020}. SREDA instead couples sample evaluations in a recursive estimator and assumes mean-squared smoothness, obtaining the improved normalized accuracy dependence $O(\kappa^3\epsilon^{-3})$ on a compact dual domain \citep{LuoEtAl2020}. The separation between bounded-variance and mean-squared-smooth access is also central to stochastic nonconvex lower bounds \citep{Arjevani2019}. Our oracle assumes smoothness only of the mean objective and supplies a new gradient estimate at each query.

Alternating updates provide another route without nested maximization. \citet{BotBohm2023} analyze stochastic alternating proximal-gradient steps for composite nonconvex--concave problems and allow unbounded dual domains in the strongly concave setting. Their guarantees retain an initial dual-distance term. \citet{YangOrvieto2022} instead stabilize alternating GDA with a moving quadratic regularizer. Under the weaker dual PL condition, their stochastic method has a leading $O(\kappa^2\epsilon^{-4})$ cost without large minibatches. The additional $\widetilde O(\kappa^5\epsilon^{-2})$ primal-gradient conversion explains the small-accuracy restriction in \Cref{tab:comparison}; the leading term alone is not a uniform bound in both $\kappa$ and $\epsilon$.

SAPD+ addresses weakly-convex--concave composite problems through regularized saddle subproblems \citep{SAPDPlus2022}. In the strongly concave case, its Moreau-envelope guarantee costs $O(\kappa L\mathcal G_0(\epsilon^{-2}+\sigma^2\epsilon^{-4}))$ calls, where $\mathcal G_0$ is an initial primal--dual gap. It also converts this guarantee to a generalized gradient mapping, which is the true primal gradient when the primal nonsmooth term vanishes, with an additional solve of lower order in $\epsilon^{-1}$. Compact domains are not required in its extension to unbounded domains. SGDA-B uses backtracking to obtain guarantees without knowing the smoothness, strong-concavity, or variance parameters \citep{XuEtAl2024}. High-probability results under light-tailed noise and dual Polyak--\L{}ojasiewicz structure \citep{LaguelEtAl2024}, and methods beyond Lipschitz smoothness \citep{GaoLiu2026}, address complementary assumptions.

\subsection{Stochastic saddle solvers and gradient criteria}
\citet{ChenLuo2024} develop recursive anchored iteration to make the joint gradient small in stochastic convex--concave and SCSC problems. \citet{LanLi2023} develop a stochastic Catalyst framework for convex--strongly-concave and SCSC optimization. These works illustrate how regularization separates deterministic contraction from statistical accuracy. The distinction between function accuracy and gradient accuracy is also fundamental in stochastic convex optimization \citep{Foster2019}.

Our outer method requires a particular SCSC guarantee: an expected global saddle gap on an unconstrained domain, random warm starts controlled in expectation, and an initialization cost that does not multiply the statistical term. We prove this interface using finite multilevel estimation, acceleration robust to conditional bias, and recursive regularization. An expected-work guarantee is converted once, at the end of the outer algorithm, to the fixed budget used in our minimax definition. This identifies the extra interface and accounting needed beyond quoting a value-gap or joint-gradient theorem.

\subsection{Randomized lower bounds and neighboring minimax classes}
The oracle-complexity framework originates in information-based lower bounds \citep{NemirovskiYudin1983}. Modern nonconvex constructions use zero chains and random embeddings to hide progress from general first-order algorithms \citep{CarmonEtAl2020,CarmonEtAl2021,Arjevani2019}. These techniques motivate our treatment of unrestricted queries, but the NC--SC geometry requires preserving strong concavity in the designated dual variable. The stochastic construction must also define an unbiased bounded-variance oracle at every query, including queries outside the region in which the chain argument is informative.

In neighboring settings, \citet{PanLi2026} establish a deterministic lower bound with linear condition-number dependence under a dual Polyak--\L{}ojasiewicz condition. \citet{WuGuYang2026} study deterministic and bounded-variance stochastic lower bounds for nonconvex--concave minimax optimization with a bounded dual domain and Moreau-envelope stationarity. Those results concern different geometry or output criteria. Here the dual domain is unconstrained, the dual objective is strongly concave, and both bounds concern the true primal gradient under the same fresh-gradient oracle.
\section{Model and Main Results}\label{sec:model}

We translate the designated initial point to $(x_0,y_0)=(0,0)$. Dimensions are not fixed globally: the problem class is the union over all finite $d_x,d_y$, and a lower-bound construction may choose a finite dimension depending on the candidate query horizon and the problem parameters.

\subsection{Localized NC--SC instances}

\begin{definition}[Fresh-gradient localized NC--SC class]\label{def:paper-class}
For parameters $L\ge\mu>0$, $\Delta,\sigma,G\ge0$, the class $\cF(L,\mu,\Delta,\sigma,G)$ consists of pairs $(f,\mathcal O)$ satisfying the following.
\begin{enumerate}
    \item $f:\R^{d_x}\times\R^{d_y}\to\R$ is differentiable and jointly $L$-smooth:
    \[
       \|\nabla f(z)-\nabla f(z')\|\le L\|z-z'\|\qquad\forall z,z'.
    \]
    \item For every $x$, the map $y\mapsto f(x,y)$ is $\mu$-strongly concave.
    \item With $\Phi(x)=\max_y f(x,y)$,
    \[
       \Phi(0)-\inf_x\Phi(x)\le\Delta,
       \qquad \|\nabla_y f(0,0)\|\le G.
    \]
    \item The oracle is a Markov kernel fixed before the interaction. At query $z_t=(x_t,y_t)$, it returns a fresh vector $\mathsf G_t$ such that, conditionally on the entire past transcript,
    \[
       \E[\mathsf G_t\mid\mathcal F_{t-1},z_t]=\nabla f(z_t),
       \qquad
       \E[\|\mathsf G_t-\nabla f(z_t)\|^2\mid\mathcal F_{t-1},z_t]\le\sigma^2.
    \]
\end{enumerate}
There is no finite-sum index, same-sample multi-point evaluation, proximal oracle, or Hessian-vector product.
\end{definition}

Let $y^\star(x)=\argmax_y f(x,y)$. Strong concavity and smoothness imply
\begin{equation}\label{eq:localization-equivalence-main}
    \mu\|y_0-y^\star(x_0)\|
    \le \|\nabla_y f(x_0,y_0)\|
    \le L\|y_0-y^\star(x_0)\|.
\end{equation}
Thus residual and distance localization are equivalent up to factors of $\kappa=L/\mu$ inside logarithms.

\subsection{Unrestricted randomized algorithms and complexity}

\begin{definition}[Algorithm class]\label{def:paper-alg}
An algorithm has an internal random seed independent of the instance and oracle randomness. At time $t$, its query is an arbitrary measurable function of the seed and the preceding query--response transcript. Its final primal output is another arbitrary measurable function of the full transcript. Queries may be dense, nonlinear, and of unbounded norm. No zero-respecting, linear-span, Krylov, obliviousness, or determinism restriction is imposed.
\end{definition}

\begin{definition}[Localized minimax complexity]\label{def:paper-complexity}
$\Tloc(L,\mu,\Delta,\sigma,G,\epsilon)$ is the least fixed integer $N$ for which some algorithm in \Cref{def:paper-alg}, using at most $N$ oracle calls, satisfies
\[
    \Pp\bigl(\|\nabla\Phi(\widehat x)\|\le\epsilon\bigr)\ge\frac23
\]
on every instance in $\cF(L,\mu,\Delta,\sigma,G)$ and every finite dimension. If no such $N$ exists, the complexity is $+\infty$.
\end{definition}

The fixed-budget requirement is important. An algorithm with an expected running time is not automatically admissible; our upper proof ends with a global truncation that converts expected total work to a deterministic call budget.

For $\epsilon>0$, write $\kappa=L/\mu$ and $\log_+u=\log\max\{1,u\}$ for $u\ge0$. Define
\begin{equation}\label{eq:ABCD-main}
\begin{aligned}
    A&:=\frac{\sqrt\kappa L\Delta}{\epsilon^2},
    &B&:=\frac{\kappa L\Delta\sigma^2}{\epsilon^4},\\
    C&:=\frac{\kappa^2\sigma^2}{\epsilon^2},
    &D_G&:=\sqrt\kappa\log_+\!\left(\frac{G}{\epsilon\sqrt\kappa}\right).
\end{aligned}
\end{equation}

\subsection{Sharp localized complexity}

\begin{theorem}[Sharp localized complexity]\label{thm:paper-main}
There are universal constants $c,C_0,C_1,\kappa_0>0$ such that, whenever $\kappa\ge\kappa_0$ and $L\Delta\ge C_0\epsilon^2$,
\begin{equation}\label{eq:paper-main-lower}
    \Tloc(L,\mu,\Delta,\sigma,G,\epsilon)
    \ge c\,(A+B+C+D_G).
\end{equation}
Conversely, let
\[
    \Lambda
    =1+\log(2+\kappa)
      +\log\!\left(2+\frac{L\Delta}{\epsilon^2}\right)
      +\log\!\left(2+\frac{\kappa^2\sigma^2}{\epsilon^2}\right).
\]
There is a fixed-budget fresh-gradient algorithm such that
\begin{equation}\label{eq:paper-main-upper}
\Tloc
\le C_1\Lambda^{10}\left[
 A+B+C+
 \sqrt\kappa\log\!\left(
 2+\frac{\kappa^2G^2}{\epsilon^2}
   +\frac{\kappa^2L\Delta}{\epsilon^2}
 \right)
\right].
\end{equation}
The prefactor $\Lambda^{10}$ is independent of $G$. Consequently,
\[
    \Tloc=\widetilde\Theta(A+B+C+D_G).
\]
\end{theorem}

The theorem is a minimax equivalence, not merely an algorithm analysis: the lower bound holds for the full class of \Cref{def:paper-alg}. Its four terms are obtained from separate admissible hard distributions. This is sufficient because the worst-case complexity is at least the maximum of the four lower bounds, and the maximum is within a universal factor of their sum.

\begin{corollary}[Noise-free endpoint]\label{cor:paper-deterministic}
Under the same nontrivial-regime conditions,
\[
    \Tloc(L,\mu,\Delta,0,G,\epsilon)
    =\widetilde\Theta\!\left(
       \frac{\sqrt\kappa L\Delta}{\epsilon^2}
       +\sqrt\kappa\log_+\frac{G}{\epsilon\sqrt\kappa}
    \right).
\]
Both lower-bound mechanisms use exact gradients, and the upper solver has a deterministic branch; the conclusion is not obtained by taking $\sigma\downarrow0$ in a noise-dependent construction.
\end{corollary}

\subsection{Infinite complexity without dual localization}

Let $\Tall(L,\mu,\Delta,\sigma,\epsilon)$ denote the same minimax complexity when the class places no restriction on $\|\nabla_y f(0,0)\|$.

\begin{theorem}[Infinite complexity without dual localization]\label{thm:paper-infinite}
For every $L\ge4\mu>0$, $\epsilon>0$, $\sigma\ge0$, and $\Delta\ge80\epsilon^2/L$,
\[
    \Tall(L,\mu,\Delta,\sigma,\epsilon)=+\infty.
\]
More precisely, for every finite horizon $N$, there is a fixed distribution over finite-dimensional, exact-oracle instances satisfying the same $(L,\mu,\Delta)$ bounds such that every unrestricted randomized algorithm using at most $N$ calls has success probability strictly below $2/3$ on average over that distribution.
\end{theorem}

The theorem concerns a uniform dimension-free complexity function. It does not claim that an individual fixed finite-dimensional problem cannot be solved in finitely many queries.

\section{Upper Bound: Proof Architecture}\label{sec:upper-overview}

The upper bound has two layers. A one-sided proximal outer loop reduces NC--SC stationarity to a sequence of SCSC saddle problems. The inner solver handles those subproblems using only fresh one-point joint gradients, with an additive initialization logarithm and the optimal statistical term up to polylogarithms.

\subsection{One-sided proximal reduction}

For a center $c\in\R^{d_x}$, define
\begin{equation}\label{eq:main-prox-subproblem}
    g_c(u,v):=f(u,v)+L\|u-c\|^2,
    \qquad P_c(u):=\max_v g_c(u,v)=\Phi(u)+L\|u-c\|^2.
\end{equation}
Because $\Phi$ is $L$-weakly convex, $P_c$ is $L$-strongly convex. Let
\[
    u_c^*:=\argmin_u P_c(u),
    \qquad m(c):=2L(c-u_c^*).
\]
The optimality condition and strong convexity give
\begin{equation}\label{eq:main-moreau-identities}
    m(c)=\nabla\Phi(u_c^*),
    \qquad
    P_c(c)-P_c(u_c^*)\ge\frac{\|m(c)\|^2}{8L}.
\end{equation}
Thus an approximate saddle solution of $g_c$ yields descent in $\Phi$ whenever the proximal gradient $m(c)$ is large.

\begin{algorithm}[t]
\caption{High-level localized fresh-gradient method $\proxouter$}
\label{alg:outer}
\begin{algorithmic}[1]
\Require Initial point $(x_0,y_0)=(0,0)$; coarse accuracy $\eta_c$; fine accuracy $\eta_f$; number of outer stages $T$.
\For{$t=0,\ldots,T-1$}
    \State Approximately solve the SCSC problem $\min_u\max_v g_{x_t}(u,v)$ to expected saddle gap $\eta_c$ using the previous primal--dual output as a warm start.
    \State Let $(x_{t+1},y_{t+1})$ be the returned point.
\EndFor
\State Draw $R$ uniformly from $\{0,\ldots,T-1\}$.
\State Starting from $(x_{R+1},y_{R+1})$, refine the same subproblem $g_{x_R}$ to expected saddle gap $\eta_f$.
\State \Return the refined primal point $\widehat x$.
\end{algorithmic}
\end{algorithm}

Choose
\[
    T=\Theta\!\left(1+\frac{L\Delta}{\epsilon^2}\right),
    \qquad
    \eta_c=\Theta\!\left(\frac{\epsilon^2}{L}\right),
    \qquad
    \eta_f=\Theta\!\left(\frac{\epsilon^2}{\kappa L}\right).
\]
If each coarse subproblem is solved to gap $\eta_c$, telescoping \eqref{eq:main-moreau-identities} gives
\[
    \frac1T\sum_{t<T}\E\|m(x_t)\|^2=O(\epsilon^2),
    \qquad
    L\sum_{t<T}\E\|x_{t+1}-x_t\|^2=O(\Delta).
\]
The random fine solve converts a small proximal gradient into a small true primal gradient through the exact identity
\begin{equation}\label{eq:main-gradient-identity}
    \nabla\Phi(\widehat x)
    =m(x_R)+\nabla P_{x_R}(\widehat x)
      -2L(\widehat x-u_{x_R}^*).
\end{equation}

\subsection{A fresh-gradient SCSC primitive}

The main algorithmic ingredient is the following result, proved in full in \Cref{sec:uppernew}.

\begin{theorem}[Localized fresh-gradient SCSC solver; informal form]\label{thm:paper-inner}
Let $g(u,v)$ be jointly $\ell$-smooth, $\alpha$-strongly convex in $u$, and $\mu$-strongly concave in $v$, where $\alpha\ge\ell/4$ and $\kappa_g=\ell/\mu$. Suppose
\[
    \E\bigl[\ell\|u_0-u^*\|^2+\mu\|v_0-v^*\|^2\bigr]\le\mathcal E_0.
\]
Using only fresh one-point joint gradients, one can return $(u,v)$ with $\E\Gap_g(u,v)\le\eta$ at expected original-call cost
\[
    \widetilde O\!\left[
        \sqrt{\kappa_g}\log\!\left(2+\frac{\kappa_g\mathcal E_0}{\eta}\right)
        +\frac{\sigma^2}{\mu\eta}
    \right],
\]
where the hidden logarithms depend only on $\kappa_g$ and $\sigma^2/(\mu\eta)$, not on $\mathcal E_0$.
\end{theorem}

The outer loop uses three features of this guarantee: it controls the joint saddle gap; it allows random subproblems and random warm starts controlled only in expectation; and its statistical term is not multiplied by the cold-start radius.

\subsection{Finite multilevel estimation with relative bias}

For fixed $v$, let
\[
    u^*(v)=\argmin_u g(u,v),
    \qquad q(v)=-\min_u g(u,v).
\]
Then $q$ is $\mu$-strongly convex and $O(\ell)$-smooth, but the original oracle does not directly return $\nabla q(v)=-\nabla_v g(u^*(v),v)$. The inner solver therefore builds an effective estimator by approximately solving the fast $u$-problem at geometrically increasing batch levels.

For a reference point $\bar u$, a target value accuracy $e$, and a regularized dual objective with strong-convexity parameter $m$, the estimator $Y(v)$ satisfies the relative bounds
\begin{align}
    \|\E Y(v)-\nabla Q(v)\|^2
    &\le \rho m e+\rho m^2\|\bar u-u^*(v)\|^2,\label{eq:main-relative-bias}\\
    \E\|Y(v)-\E Y(v)\|^2
    &\le \widetilde O(\sigma^2)+\rho m^2\|\bar u-u^*(v)\|^2.\label{eq:main-relative-var}
\end{align}
All levels use independent fresh samples. The randomized telescoping correction is finite, so the construction has an almost surely finite number of calls and does not require evaluating one sample at two points.

The key feature is \emph{relative} rather than absolute bias. The distance term in \eqref{eq:main-relative-bias} can be absorbed into a contracting accelerated potential. The level cutoff depends on the target accuracy and noise level, but not on the unknown initial radius; this is what prevents a factor $\log G$ from multiplying the stochastic cost.

\subsection{Robust acceleration and recursive regularization}

For an $H$-smooth, $m$-strongly convex objective $Q$, the accelerated recurrence used in the proof satisfies
\begin{equation}\label{eq:main-robust-agd}
    \E[\mathcal P_{t+1}\mid\mathcal F_t]
    \le \left(1-\frac12\sqrt{\frac mH}\right)\mathcal P_t
       +\frac{V_t}{H}
       +\frac{2}{\sqrt{mH}}\|b_t\|^2,
\end{equation}
where $b_t$ and $V_t$ are the conditional bias and centered variance of the effective gradient. Combining \eqref{eq:main-relative-bias}--\eqref{eq:main-relative-var} with epochwise batching halves the expected energy in each epoch. The deterministic work sums geometrically to $\widetilde O(\sqrt{H/m}\log(2+E_0/e))$, while the stochastic work sums as
\[
    \sum_s \frac{\sigma^2}{mE_s}=O\!\left(\frac{\sigma^2}{me}\right)
\]
because the epoch envelopes $E_s$ decrease geometrically.

A recursive regularization chain then converts value accuracy for a sequence of increasingly strongly convex dual objectives into a small gradient of the original $q$. A final fast-variable solve and the standard residual certificate
\[
    \Gap_g(u,v)
    \le \frac{\|\nabla_u g(u,v)\|^2}{2\alpha}
        +\frac{\|\nabla_v g(u,v)\|^2}{2\mu}
\]
produce the joint saddle-gap guarantee in \Cref{thm:paper-inner}.

\subsection{Complexity accounting and fixed-budget conversion}

The $T=\Theta(L\Delta/\epsilon^2)$ coarse solves contribute
\[
    \widetilde O(A+B).
\]
The single fine solve contributes
\[
    \widetilde O(\sqrt\kappa+C).
\]
The first subproblem has initial energy
\[
    O\!\left(\frac{G^2}{\mu}+\kappa\Delta\right),
\]
which generates the displayed cold-start logarithm in \eqref{eq:paper-main-upper}. Strong monotonicity of the saddle operator and the outer descent bound control every later warm start without additional dependence on $G$.

The analysis first yields a known upper bound $\mathcal B$ on the expected \emph{total number of original oracle calls}, including all random multilevel costs. The executable algorithm is stopped after $\lceil 12\mathcal B\rceil$ calls and outputs zero if the limit is exceeded. Markov's inequality bounds the truncation probability by $1/12$; the final squared-gradient guarantee gives another failure probability at most $1/12$. Hence the fixed-budget algorithm succeeds with probability at least $5/6$, establishing \eqref{eq:paper-main-upper}.
\section{Lower Bounds: Four Hard Mechanisms}\label{sec:lower-overview}

The lower bound is proved by four separate fixed distributions. Each construction is globally admissible, protects against dense and unbounded queries, and includes the final output in the hiding argument. This section explains the common lifting device and the mechanism behind each term; complete formulas and constants appear in \Cref{sec:D,sec:B,sec:C,sec:Anew}.

\subsection{A common lifting principle}

The intrinsic hard cores are finite primal--dual chains. Their primal coordinates and dual coordinates are embedded through independent random orthonormal frames
\[
    U\in\St(d_x,D_x),
    \qquad V\in\St(d_y,D_y).
\]
Separate frames are essential: a generic rotation mixing $x$ and $y$ need not preserve strong concavity in the designated dual variable.

A direct random-rotation proof faces two obstacles.
\begin{enumerate}
    \item An unrestricted algorithm may query a dense vector with a nonzero component in every hidden direction. Smooth chains generally leak infinitesimal information through these components.
    \item Spherical-cap bounds control the projection of a bounded query, but query norms are unrestricted.
\end{enumerate}
We address the first issue by replacing informative links with functions that are exactly flat below a threshold. We address the second by modifying the objective itself: outside a protected core, a convex radial term dominates and makes the informative part independent of the hidden frame. No oracle-side clipping is used.

The filtration argument uses an auxiliary interaction defined on every trajectory. At each step it exposes only a prefix of the hidden frames and computes a reply from those exposed columns, the algorithm seed, and independent auxiliary coins. Conditioned on this information, the unexposed frame completion is uniform on the appropriate orthogonal complement. The true and auxiliary transcripts coincide until a threshold projection is violated. We bound the probability of such a violation under the \emph{unstopped} auxiliary law and then couple the two interactions. We never condition on the absence of previous violations and subsequently claim an unconditioned Haar law.

\subsection{Dual localization and unlocalized impossibility}

The localization family is a two-ended dual path. A source at the far endpoint propagates through a strongly convex tridiagonal dual quadratic; a hidden sign at the central edge changes the correct primal stationary interval. Threshold-flat edge energies ensure that, before the chain fronts meet, the transcript is independent of the sign even for dense queries. A radial convex cap removes hidden-frame dependence for large dual queries.

For a path of length $m=\Theta(N)$, the endpoint Green function decays as
\[
    (H^{-1})_{1m}
    \asymp \frac{1}{\mu\sqrt\kappa}\exp\!\left(-\Theta\!\left(\frac{m}{\sqrt\kappa}\right)\right).
\]
To make the hidden sign shift the primal derivative by $\Theta(\epsilon)$, the initial dual source must therefore have size
\[
    G\asymp \epsilon\sqrt\kappa\exp\!\left(\Theta\!\left(\frac{m}{\sqrt\kappa}\right)\right).
\]
Solving for $m$ yields
\[
    N=\Omega\!\left(\sqrt\kappa\log_+\frac{G}{\epsilon\sqrt\kappa}\right).
\]
If $G$ is unrestricted, the source magnitude can be chosen for any proposed finite horizon, proving \Cref{thm:paper-infinite}. The oracle is exact, so this mechanism also covers the deterministic endpoint.

\subsection{Global stochastic progress}

The global stochastic family starts from a smooth nonconvex zero chain of length
\[
    T=\Theta\!\left(\frac{L\Delta}{\epsilon^2}\right).
\]
Each link is represented as a square and lifted through a strongly concave dual coordinate. The full objective has zero initial dual residual and a primal stationarity barrier until the last chain coordinate is reached.

For a query $z$, let $f_T$ be the full instance and let $f_{k(z)}$ be a truncated instance determined only by the current query and the fixed hidden frames. The construction proves a uniform global residual bound
\[
    \|\nabla f_T(z)-\nabla f_{k(z)}(z)\|
    \le K_R a_s,
    \qquad a_s^2=\Theta\!\left(\frac{\epsilon^2}{\kappa}\right).
\]
The oracle returns
\begin{equation}\label{eq:overview-stochastic-oracle}
    \mathsf G(z)
    =\nabla f_{k(z)}(z)
     +\frac{\zeta}{p}\bigl(\nabla f_T(z)-\nabla f_{k(z)}(z)\bigr),
    \qquad \zeta\sim\operatorname{Bernoulli}(p),
\end{equation}
where
\[
    p=\min\!\left\{1,\Theta\!\left(\frac{\epsilon^2}{\kappa\sigma^2}\right)\right\}.
\]
The coin is independent of the query and hidden frames. Consequently \eqref{eq:overview-stochastic-oracle} is a fixed Markov kernel satisfying exact conditional unbiasedness and the variance bound at every query. On the coupled good event, a successful coin reveals at most one new logical stage. Therefore fewer than $\Omega(T/p)$ calls fail with constant probability, giving
\[
    \Omega\!\left(
        \frac{L\Delta}{\epsilon^2}
        +\frac{\kappa L\Delta\sigma^2}{\epsilon^4}
    \right).
\]
The first summand is dominated by the accelerated deterministic term in the final theorem, while the second is $B$.

\subsection{Local true-gradient identification}

The $C$ term does not require a long chain. We construct two globally smooth one-primal/one-dual instances indexed by a sign $s\in\{-1,+1\}$. Their initial dual residuals are zero, and the sets of primal outputs satisfying $|\Phi_s'(x)|\le\epsilon$ are disjoint. Yet the two oracle means differ only in a clipped dual component, uniformly over every possible adaptive query:
\[
    \|\nabla f_+(x,y)-\nabla f_-(x,y)\|
    =O\!\left(\frac{\epsilon}{\kappa}\right).
\]
With Gaussian noise of variance $\sigma^2$ in that component, each query contributes at most
\[
    O\!\left(\frac{\epsilon^2}{\kappa^2\sigma^2}\right)
\]
conditional KL divergence. The adaptive chain rule for KL divergence and Pinsker's inequality imply
\[
    N=\Omega\!\left(\frac{\kappa^2\sigma^2}{\epsilon^2}\right).
\]
This family is already low-dimensional and the argument applies directly to arbitrary adaptive randomized algorithms.

\subsection{Accelerated deterministic progress}

The exact-oracle term $A$ requires a different construction from the stochastic term. A stochastic masking coin cannot help when $\sigma=0$, and a naive rotation of a bilinear dual chain leaks information through arbitrarily small off-chain projections.

Our intrinsic construction interleaves a nonlinear primal chain with dual subchains of length
\[
    n=\Theta(\sqrt\kappa).
\]
Each dual subchain uses a Robin-type path matrix whose endpoint inverse entries remain of order $n$. After eliminating the dual variables, one obtains a constant-strength coupling between consecutive primal stages. The edge potential
\[
    \varphi_{\delta,b}'(t)
    =\operatorname{sgn}(t)\min\{(|t|-\delta)_+,b\}
\]
is exactly flat for $|t|\le\delta$ and has globally bounded slope. Flatness gives exact transcript locality, while the bounded slope preserves the $\sqrt\kappa$ dependence under primal compression and dual protection.

The construction has $T=\Theta(L\Delta/\epsilon^2)$ primal stages, each requiring the traversal of $\Theta(n)$ hidden vertices. Under the auxiliary product-Haar interaction, one oracle call exposes at most one vertex. Appending the final output as a dummy query ensures that an unqueried output cannot guess the terminal coordinate. Hence
\[
    N=\Omega(nT)
     =\Omega\!\left(\frac{\sqrt\kappa L\Delta}{\epsilon^2}\right).
\]
The family has $G_0=0$ and uses exact gradients, so it belongs to every larger localization and variance budget.

\subsection{Synthesis}

The four hard distributions prove
\[
    \Tloc\ge\Omega(A),\qquad
    \Tloc\ge\Omega(B),\qquad
    \Tloc\ge\Omega(C),\qquad
    \Tloc\ge\Omega(D_G).
\]
Taking the maximum and using $\max_i a_i\ge(\sum_i a_i)/4$ yields the lower half of \Cref{thm:paper-main}. The constructions also establish all endpoint claims: the $A$ and $D_G$ families are exact-oracle, and the $A,B,C$ families satisfy $G_0=0$.
\section{Discussion}\label{sec:discussion}

The characterization separates four costs that a single accuracy exponent does not describe. Deterministic propagation and global stochastic progress scale with the initial primal gap. Local identification is paid once to resolve the true primal gradient, and dual localization contributes an additive logarithm. The unlocalized impossibility shows why an initialization parameter can be essential even when it appears only inside a logarithm.

The oracle model is central to this conclusion. Each response is a fresh unbiased joint-gradient estimate with bounded conditional variance; no regularity of individual sample gradients is assumed. Common-sample differences or finite-sum access can support different complexity guarantees. Likewise, a primal--dual gradient mapping or Moreau-envelope criterion should not be substituted for the true primal gradient without accounting for the conversion cost.

The matching statement holds in the nontrivial regime $\kappa\ge\kappa_0$ and $L\Delta\ge C_0\epsilon^2$. It is a dimension-free oracle characterization, with lower-bound dimensions allowed to grow with the query horizon. It does not give a sharp fixed-dimensional rate or an arithmetic-complexity bound. At $\Delta=0$, the initial primal point is already optimal, so no positive query lower bound is possible. The fixed-confidence result also leaves the optimal dependence on an arbitrary failure probability open.

Reducing the polylogarithmic overhead and the dimensions of the hard instances are natural next questions. Other directions include weaker noise moment assumptions and generalized smoothness. Each extension requires revisiting both the fresh-gradient solver and the admissibility of the hard oracle.

\bibliographystyle{unsrtnat}
\bibliography{references}

\clearpage
\appendix
\section{Elementary Devices}

\subsection{A concrete clip and a convex smooth maximum}
Let $h(s)=1-3s^2+2s^3$ for $s\in[0,1]$.  Define an odd $C^2$ function $\rho_R$
by $\rho_R(0)=0$ and the even derivative
\[
 \rho_R'(t)=\begin{cases}
 1,&|t|\le R,\\
 h(|t|/R-1),&R<|t|<2R,\\
 0,&|t|\ge2R.
 \end{cases}
\]
Then
\begin{equation}\label{eq:clip}
 \rho_R(t)=t\ (|t|\le R),\quad |\rho_R(t)|\le\min\{|t|,3R/2\},\quad
 |\rho_R'|\le1,\quad |\rho_R''|\le\frac{3}{2R}.
\end{equation}
For $\eta>0$ set
\[
 \chi_\eta(t)=\begin{cases}0,&t\le0,\\ t^2/(2\eta),&0<t<\eta,\\
 t-\eta/2,&t\ge\eta,\end{cases}
 \qquad \mathcal M_\eta(A,B)=A+\chi_\eta(B-A).
\]
The function $\mathcal M_\eta$ is convex and nondecreasing in each argument.  It
is $C^1$ with a Lipschitz gradient, equals $A$ if $B\le A$, and equals
$B-\eta/2$ if $B-A\ge\eta$.  In particular $\mathcal M_\eta(A,B)\ge A$.
When $A,B$ are differentiable, write $b=\chi_\eta'(B-A)\in[0,1]$.  Almost everywhere,
\begin{equation}\label{eq:maxhess}
 \nabla^2\mathcal M_\eta(A,B)
 =(1-b)\nabla^2A+b\nabla^2B+
 \chi_\eta''(B-A)\nabla(B-A)\nabla(B-A)^\top.
\end{equation}
All smoothness assertions made using~\eqref{eq:maxhess} concern globally Lipschitz
gradients; twice continuous differentiability is not required by the problem.
An almost-everywhere Hessian bound here implies the gradient Lipschitz bound by
integration along segments (or by smooth approximation and passage to the limit).

\subsection{Spherical concentration and the correct auxiliary filtration}
\begin{lemma}[A spherical tail bound]\label{lem:sphere}
If $u$ is uniform on the unit sphere of a $d$-dimensional Euclidean space, then
for every fixed $w$ and $a>0$,
\[
 \Pp(|\ip{u}{w}|>a)\le2\exp\!\left(-\frac{da^2}{2\norm{w}^2}\right).
\]
The right-hand side is interpreted as zero when $w=0$.
\end{lemma}
\begin{proof}
By rotational invariance it suffices to consider the first coordinate of $u$.
Its odd moments vanish, and its even moments are
\[
 \E u_1^{2j}=\frac{1\cdot3\cdots(2j-1)}{d(d+2)\cdots(d+2j-2)}
 \le\frac{(2j-1)!!}{d^j}.
\]
The equality follows, for example, by writing a standard Gaussian vector as its
independent radius times its uniform direction and comparing even moments.
Expanding the exponential gives $\E e^{tu_1}\le e^{t^2/(2d)}$.
Chernoff's inequality, optimized over $t$, proves each of the two tails.
\end{proof}

\begin{lemma}[Deferred decisions for an auxiliary interaction]\label{lem:ghost}
Sample independent orthonormal frames $U,V$ uniformly before the run.  Consider
an auxiliary interaction which, at each step, exposes a prefix of each frame;
its replies and exposure decisions are functions only of the exposed columns,
the independent algorithm seed, and auxiliary coins independent of both frames.
Conditioned on this entire auxiliary information, the remaining columns are
independent uniform frames on the corresponding orthogonal complements.
Consequently Lemma~\ref{lem:sphere} applies to the next auxiliary query, with the
remaining ambient dimension in place of $d$.

If an actual interaction and this auxiliary interaction have identical replies
until a projection threshold is exceeded, the probability of a disagreement is
bounded by a union bound over threshold violations in the auxiliary interaction.
One must not condition on the absence of previous violations when invoking the
uniform-frame law.
\end{lemma}
\begin{proof}
Conditional on a specified prefix of a uniform orthonormal frame, its completion
is uniformly distributed on the orthogonal complement.  This follows from
invariance under every orthogonal transformation fixing that prefix.  The two
frames have a product law.  Any additional auxiliary information described in the
statement is a measurable function of their prefixes and independent randomness,
and therefore imposes no further restriction on either completion.  Applying
this observation successively proves the assertion, including a random exposure
length determined by the independent coins and exposed information.

For the coupling, use the same sampled frames, seed, and coins in both runs.
Before their first disagreement, their queries agree.  On a trajectory of the
auxiliary run without any listed violation, the assumed equality of the reply
rules gives equality at the first step and then at every subsequent step by
induction.  Thus disagreement is contained in the union of auxiliary violation
events.  Each such probability is estimated under the unstopped auxiliary law.
No conditional uniformity assertion is made under the survival event.
\end{proof}

\section{Exact-Oracle Localization and Unlocalized Impossibility}\label{sec:D}

\subsection{The horizon-specific instance}
Fix an integer $N\ge1$ and set
\[
 m=4N+4,\quad k=m/2,\quad a=L/4,\quad \gamma=L/8,\quad
 \beta=L/8192,\quad d_0=4\epsilon.
\]
Let $T_+$ be the $m\times m$ tridiagonal matrix with diagonal $2$ and off-diagonal
$-1$.  Put
\[
 D_s=\operatorname{diag}(\underbrace{1,\ldots,1}_{k},
                          \underbrace{s,\ldots,s}_{m-k}),\quad
 T_s=D_sT_+D_s,\quad H_s=\mu I+\beta T_s,\quad s\in\{-1,1\}.
\]
Define $\theta>0$ and $h_m>0$ by
\begin{equation}\label{eq:green}
 \cosh\theta=1+\frac{\mu}{2\beta},\qquad
 h_m=(H_+^{-1})_{1m}
      =\frac{\sinh\theta}{\beta\sinh((m+1)\theta)}.
\end{equation}
The inverse formula follows by solving the three-term recurrence
$-\beta z_{i-1}+(\mu+2\beta)z_i-\beta z_{i+1}=0$ with zero boundary values;
its two homogeneous solutions are $e^{i\theta}$ and $e^{-i\theta}$.
Set
\begin{equation}\label{eq:Dparams}
 M=\frac{2\epsilon}{\gamma h_m},\quad
 q=a+\gamma^2(H_+^{-1})_{11},\quad
 \delta=\frac{\epsilon\mu}{16\gamma\beta\sqrt{m+1}},\quad
 R=\frac{8(M+4\epsilon)}{\mu}.
\end{equation}
Let $V=(v_1,\ldots,v_m)$ be a uniform frame in $\R^{d_y}$, and put
$z_i=\ip{v_i}{y}$. With
\[
 \varphi_\delta(t)=\frac12\big((|t|-\delta)_+\big)^2,
\]
define
\begin{align*}
 Q_{s,V}(y)&=\beta\left[\varphi_\delta(z_1)
 +\sum_{i=1}^{m-1}\varphi_\delta(z_i-s_i z_{i+1})
 +\varphi_\delta(z_m)\right],\\
 s_i&=1\quad(i\ne k),\qquad s_k=s,\\
 W(y)&=4\beta(\norm y^2-R^2),\qquad
 \overline Q_{s,V}(y)=\mathcal M_{\beta R^2}(Q_{s,V}(y),W(y)).
\end{align*}
The actual objective is
\begin{equation}\label{eq:Dobj}
 f_{s,V}(x,y)=\frac a2x^2+d_0x+\gamma x\ip{v_1}{y}
       -\frac\mu2\norm y^2-\overline Q_{s,V}(y)+M\ip{v_m}{y},
 \qquad x\in\R.
\end{equation}
Its oracle returns its exact joint gradient.  The distribution consists of an
independent fair sign $s$ and the uniform frame $V$, both sampled before interaction.

\subsection{Global regularity and the primal decision problem}
\begin{lemma}\label{lem:Dregular}
For $L\ge4\mu$,~\eqref{eq:Dobj} is jointly $L$-smooth and $\mu$-strongly concave
in $y$.  It has initial dual residual $M$ and primal gap at most
$80\epsilon^2/L$.  For every $V$, its successful output set is contained in
\[
 J_+=\left[-\frac{7.5\epsilon}{q},-\frac{4.5\epsilon}{q}\right]
 \quad\hbox{or}\quad
 J_-=\left[-\frac{3.5\epsilon}{q},-\frac{0.5\epsilon}{q}\right],
\]
according to its sign.  These intervals are disjoint and independent of $V$.
\end{lemma}
\begin{proof}
Let $A_s$ be the signed path incidence matrix, including the two endpoint rows,
so that $A_s^\top A_s=T_s$ and $\norm{A_s}\le2$.
Since $\varphi_\delta$ is convex and has a $1$-Lipschitz derivative,
\[
 0\le Q\le2\beta\norm y^2,\quad
 \norm{\nabla Q}\le4\beta\norm y,\quad 0\preceq\nabla^2Q\preceq4\beta I.
\]
The cap is convex.  It agrees with $Q$ on $\norm y\le R$, and with
$W-\beta R^2/2$ on $\norm y\ge2R$.  In the transition region
$0<W-Q<\beta R^2$, the inequality $Q\le2\beta\norm y^2$ gives
$\norm y^2<(5/2)R^2$.  In~\eqref{eq:maxhess}, the base Hessian has norm at most
$8\beta$, and the rank-one term has norm at most
\[
 \frac{(12\beta\norm y)^2}{\beta R^2}\le360\beta.
\]
Thus $\nabla\overline Q$ is $368\beta$-Lipschitz.  The full Hessian norm is bounded by
\[
 \max\{a,\mu+368\beta\}+\gamma<L
\]
when $L\ge4\mu$, while its dual block is at most $-\mu I$.  At the origin all
flat-edge derivatives vanish and the cap is inactive, giving $G_0=M$.

Compare first with the uncapped, unflattened quadratic using
$\beta y^TVT_sV^Ty/2$ in place of $\overline Q$.
Its dual maximizer lies in the frame and has coordinates
$H_s^{-1}(\gamma x e_1+M e_m)$. Its primal derivative is exactly
\[
 qx+(4+2s)\epsilon,
\]
because $H_s^{-1}=D_sH_+^{-1}D_s$ and $\gamma Mh_m=2\epsilon$.
Flattening perturbs the dual gradient by at most
\[
 \beta\norm{A_s^\top}\delta\sqrt{m+1}
 \le2\beta\delta\sqrt{m+1}.
\]
Strong monotonicity of the dual minimization problem therefore moves its optimizer
by at most $2\beta\delta\sqrt{m+1}/\mu=\epsilon/(8\gamma)$.
For $x\in[-8\epsilon/q,0]$, $q\ge L/4$ implies $\gamma|x|\le4\epsilon$.
The quadratic optimizer has norm at most $(M+4\epsilon)/\mu=R/8$.
The flattened optimizer has norm less than $R/4$, so the cap is inactive there.
Since the capped objective is no larger than the uncapped one and agrees at this
optimizer, their maxima agree.  Hence the actual primal derivative satisfies
\begin{equation}\label{eq:Dderror}
 |\Phi_{s,V}'(x)-[qx+(4+2s)\epsilon]|\le\epsilon/8
 \quad\hbox{for }x\in[-8\epsilon/q,0].
\end{equation}
The primal function is globally $a$-strongly convex, as a maximum of functions
with that common primal strong-convexity modulus. At each left endpoint of $J_s$
its derivative is less than $-\epsilon$ and at each right endpoint it is greater
than $\epsilon$, by~\eqref{eq:Dderror}. Monotonicity excludes successful outputs
outside $J_s$. Finally strong convexity gives
\[
 \Phi(0)-\inf\Phi\le\frac{|\Phi'(0)|^2}{2a}
 \le\frac{2(6+1/8)^2\epsilon^2}{L}<\frac{80\epsilon^2}{L}.
\]
\end{proof}

\subsection{The unrestricted distributional argument}
\begin{theorem}[Horizon-specific localization bound]\label{thm:Dhorizon}
Take the parameters in~\eqref{eq:Dparams} and choose
\begin{equation}\label{eq:Ddimension}
 d_x=1,\qquad d_y=m+3+
 \left\lceil\frac{32R^2}{\delta^2}\log\bigl(24m(N+1)\bigr)\right\rceil.
\end{equation}
If $G\ge M$ and $\Delta\ge80\epsilon^2/L$, every unrestricted randomized
algorithm with at most $N$ calls has average success probability at most $7/12$
on the fixed distribution~\eqref{eq:Dobj}.
\end{theorem}
\begin{proof}
Give the algorithm $v_1,v_m$ for free.  Before answering call $t$, an auxiliary
interaction exposes one more column at each end of the path.  It computes its
reply after setting the still unexposed frame coordinates to zero in $Q$.
It retains the full, known norm $\norm y$ in the radial terms and cap.
The central edge is not exposed in $N$ calls, so all these auxiliary replies and
exposed columns are independent of $s$.

At a query with $\norm y>2R$, both the true and auxiliary caps are radial and
have the same gradient.  At a query with $\norm y\le2R$, impose the good event
\[
 |\ip{v_j}{y}|\le\delta/2
 \quad\hbox{for every column unexposed \emph{before} this query.}
\]
Every edge joining two such coordinates is flat, since its signed difference
has magnitude at most $\delta$.  An edge joining a previously exposed coordinate
to the next one is computed exactly after that next column is exposed.  Edges
between the newly exposed columns and still hidden columns remain flat. Endpoint
terms in still hidden coordinates are also flat. Thus $Q$, its gradient, the cap,
and the full reply agree exactly with their auxiliary versions on this event.

By Lemma~\ref{lem:ghost}, before each auxiliary query every unexposed column is
uniform on a complement of dimension at least $d_y-m$.  Lemma~\ref{lem:sphere}
with norm bound $2R$ and threshold $\delta/2$ bounds each bad projection by
\[
 2\exp\!\left(-\frac{(d_y-m)\delta^2}{32R^2}\right).
\]
A union bound over $mN$ possible projections and~\eqref{eq:Ddimension} bounds the
probability of any discrepancy by $1/12$.  These are probabilities under the
unstopped auxiliary interaction, not probabilities conditioned on survival.

The auxiliary output has a distribution independent of the fair sign, even after
including the algorithm's random seed. Since $J_+$ and $J_-$ are disjoint, its
average success is at most $1/2$. Coupling changes this by at most $1/12$.
There is no unqueried-output loophole: successful scalar outputs must lie in the
two disjoint intervals regardless of their norm, and the output is a measurable
function of the coupled transcript. This proves the claim.
\end{proof}

\begin{corollary}\label{cor:D}
For $\kappa\ge8192$ and $\Delta\ge80\epsilon^2/L$,
\[
 \Tloc\ge\frac{\sqrt\kappa}{1024}
               \logp\!\left(\frac{G}{\epsilon\sqrt\kappa}\right).
\]
Thus the localization term in Theorem~\ref{thm:paper-main} has the stated argument
of the logarithm. Theorem~\ref{thm:paper-infinite} also holds.
\end{corollary}
\begin{proof}
Only $b=G/(\epsilon\sqrt\kappa)>1$ needs consideration. The formula for
$\theta$ implies
\[
 \theta\le\frac{\sqrt{8192}}{\sqrt\kappa}\le1,\qquad
 \sinh\theta\ge\frac{\sqrt{8192}}{\sqrt\kappa},\qquad
 M_m=\frac{\epsilon}{512\sinh\theta}\sinh((m+1)\theta).
\]
Take $N=\lfloor\sqrt\kappa\log b/1024\rfloor$. Then
\[
 (4N+5)\theta\le\frac{\sqrt{8192}}{256}\log b+5,
\]
so, using $\sinh t\le e^t/2$,
\[
 \frac{M_{4N+4}}{\epsilon\sqrt\kappa}
 \le\frac{e^5}{1024\sqrt{8192}}\,
                    b^{\sqrt{8192}/256}\le b.
\]
Both the prefactor and the exponent in the last expression are less than one.
If $N\ge1$, Theorem~\ref{thm:Dhorizon} excludes $N$ calls; integer-valued
complexity is therefore greater than $\sqrt\kappa\log b/1024$.
If $N=0$, the zero-query testing family of Section~\ref{sec:C} gives the needed
one-call lower bound. Its gap is at most $151424\epsilon^2/(\kappa L)$,
which is less than $80\epsilon^2/L$ for $\kappa\ge8192$.

For the unlocalized assertion, fix any $N\ge1$, with any $L\ge4\mu$.
The number $M$ in~\eqref{eq:Dparams} is finite.  With no localization restriction
it is admissible, and~\eqref{eq:Ddimension} is a finite integer.  Theorem~\ref{thm:Dhorizon}
therefore excludes this arbitrary horizon. The exact oracle has variance zero
and is admissible for every variance budget, including $\sigma=0$.
\end{proof}

\section{A Global Stochastic Lower Bound with Zero Dual Residual}\label{sec:B}

\subsection{A bounded square-link core}
Define
\[
 \Psi(t)=\begin{cases}0,&t\le1/2,\\
 \exp\bigl(1-(2t-1)^{-2}\bigr),&t>1/2,\end{cases}
 \quad \Theta(t)=\int_{-\infty}^t e^{-s^2/2}\,ds,\quad C_\Theta=\sqrt{2\pi}.
\]
The scalar chain below is the standard threshold construction used in
nonconvex oracle lower bounds~\cite{Arjevani2019}. For $u\in\R^T$, let
\begin{align}
 F_T(u)&=-\Theta(u_1)+\sum_{i=2}^T
  [\Psi(-u_{i-1})\Theta(-u_i)-\Psi(u_{i-1})\Theta(u_i)],\label{eq:F}\\
 D_T(u)&=-\Theta(u_1)-C_\Theta\sum_{j=1}^{T-1}\Psi(u_j),\nonumber\\
 s_{i,+}(u)&=\sqrt{\Psi(-u_{i-1})\Theta(-u_i)},\quad
 s_{i,-}(u)=\sqrt{\Psi(u_{i-1})(C_\Theta-\Theta(u_i))}.
 \nonumber
\end{align}
The identity
\begin{equation}\label{eq:squares}
 F_T=D_T+\sum_{i=2}^T(s_{i,+}^2+s_{i,-}^2)
\end{equation}
is exact. Each link depends on two adjacent coordinates.

\begin{lemma}[Core bounds]\label{lem:corebounds}
The links are $C^2$ and obey the dimension-independent component bounds
\[
 |s_{i,\pm}|\le4=:S,\qquad
 |\partial_j s_{i,\pm}|\le20,\qquad
 |\partial_j\partial_k s_{i,\pm}|\le1000.
\]
In addition
\[
 |F_T|\le 9T,\quad \norm{\nabla F_T}\le10^6\sqrt T,\quad
 \norm{\nabla^2F_T}\le10^6,
\]
while $\norm{\nabla D_T}\le49\sqrt T$ and
$\norm{\nabla^2D_T}\le1537$.  If $|u_T|\le1/4$, there is a coordinate $j$ with
$|u_j|\le1$ and $\partial_jF_T(u)\le-c_0$, where $c_0=e^{-1/2}$.
\end{lemma}
\begin{proof}
Writing $p=\sqrt\Psi$, direct differentiation with $v=(2t-1)^{-1}$ gives
\[
 p'=2\sqrt e\,v^3e^{-v^2/2},\quad
 p''=4\sqrt e\,(v^6-3v^4)e^{-v^2/2}
 \quad(t>1/2).
\]
These derivatives tend to zero at the threshold, and satisfy
$|p|<2$, $|p'|<4$, $|p''|<120$ by maximizing powers times a Gaussian.
For $b(t)=\sqrt{\Theta(t)}$, $|b|<2$, $|b'|\le10$, and $|b''|\le100$ suffice.
For completeness, when $t=-r\le-1$, integration by parts gives
$\Theta(-r)\ge r e^{-r^2/2}/(r^2+1)\ge e^{-r^2/2}/(2r)$.
The formulas
$b'=\Theta'/(2\sqrt\Theta)$ and
$b''=\Theta''/(2\sqrt\Theta)-(\Theta')^2/(4\Theta^{3/2})$
then bound the derivatives by constants times $r^{3/2}e^{-r^2/4}$.
On $[-1,0]$ use $\Theta\ge e^{-1/2}/2$; on $[0,\infty)$ use $\Theta\ge1$.
The reflected tail $\sqrt{C_\Theta-\Theta(t)}=b(-t)$ obeys the same bounds.
The product rule proves the displayed link bounds.

Each row of the link Jacobian has at most two nonzero entries of magnitude $20$,
and each column has at most four. Its operator norm is at most $\sqrt8\,20$.
Moreover $\norm{\sum_\ell w_\ell\nabla^2s_\ell}\le8000\norm w_\infty$.
Using $\Psi=p^2$, the bounds $|\Psi'|\le16$ and $|\Psi''|\le512$ give the
bounds on $D_T$. Identity~\eqref{eq:squares} then gives the stated, deliberately
loose, gradient and Hessian bounds for $F_T$. Its value bound follows directly
from $0\le\Psi<3$, $0\le\Theta<3$, and the fact that $\Psi(t)$ and $\Psi(-t)$
cannot both be positive.

All partial derivatives of $F_T$ are nonpositive. Choose the first $j$ with
$|u_j|<1$, interpreting $u_0=1$. The incoming term contributes at most $-c_0$:
for $j=1$ this is $-\Theta'(u_1)$; otherwise one of
$\Psi(\pm u_{j-1})$ is at least $\Psi(1)=1$.  The outgoing term is nonpositive.
Such a $j$ exists because $|u_T|\le1/4$.
\end{proof}

Let $V$ have $2(T-1)$ orthonormal columns $v_{i,\pm}$.  Choose positive $\lambda,r$
later, and put
\[
 a_s=2\sqrt{\mu\lambda},\quad R_0=8S\sqrt{\lambda/\mu},\quad
 H_0=48S^2T\lambda=768T\lambda,\quad R_y=2\sqrt{H_0/\mu}.
\]
For any subset of consecutive links $2,\ldots,j$, define
\begin{align*}
 \Gamma_j(u,y)&=a_s\sum_{i=2}^j\sum_{\nu\in\{+,-\}}
                    s_{i,\nu}(u)\rho_{R_0}(\ip{v_{i,\nu}}{y}),\\
 Q_j(u,y)&=\frac\mu2\norm y^2-\Gamma_j(u,y),\qquad
 W(y)=2\mu\norm y^2-4H_0,\\
 K_j(u,y)&=\mathcal M_{H_0}(Q_j(u,y),W(y)).
\end{align*}
For $j=T$ write $\Gamma,Q,K$ without the subscript. Crucially, the full ambient
norm $\norm y$, not a hidden-frame projection norm, is used in every radial term.
The common $H_0,W$ are used for all truncations.

\begin{lemma}[Dual protection and exact maximization]\label{lem:Bcap}
The function
\[
 f_{\rm core}(u,y)=\lambda D_T(u)-\frac\mu2\norm y^2-K_T(u,y)
\]
is $\mu$-strongly concave in $y$, and
\begin{equation}\label{eq:Bvalue}
 \max_y f_{\rm core}(u,y)=\lambda F_T(u).
\end{equation}
For $\norm y\ge R_y$, every $K_j$ equals $W-H_0/2$ and is independent of all
link frames and $u$. Uniformly in $j,T,u,y$,
\begin{align}
 \norm{\nabla_u f_{\rm core}}&\le8000\lambda\sqrt T,&
 \norm{\nabla^2_{uu}f_{\rm core}}&\le10^6\lambda,\nonumber\\
 \norm{\nabla^2_{uy}f_{\rm core}}&\le2000\sqrt{\mu\lambda},&
 \norm{\nabla^2_{yy}f_{\rm core}}&\le50\mu.\label{eq:Bblock}
\end{align}
The same bounds hold for truncated seed and link sums defined below.
\end{lemma}
\begin{proof}
The clip and link bounds imply
\begin{align*}
 |\Gamma_j|&\le H_0,&\norm{(\Gamma_j)_u}&\le7680\lambda\sqrt T,&
 \norm{(\Gamma_j)_{uu}}&\le768000\lambda,\\
 \norm{(\Gamma_j)_y}&\le12\sqrt{\mu\lambda T},&
 \norm{(\Gamma_j)_{uy}}&\le114\sqrt{\mu\lambda},&
 \norm{(\Gamma_j)_{yy}}&\le3\mu/8.
\end{align*}
Thus $Q_j$ is convex in $y$, and so is $K_j$, proving strong concavity.
The uncapped objective is
$\lambda D_T-\mu\norm y^2+\Gamma_T$.
Its unique maximizing frame coordinates are
$w_{i,\nu}^*=\sqrt{\lambda/\mu}\,s_{i,\nu}(u)$, with zero orthogonal component.
They lie in the clip's linear region. At that point
$\mu\norm{y^*}^2\le2S^2T\lambda$ and
$\Gamma_T(u,y^*)=2\mu\norm{y^*}^2$, so $W-Q_T<0$.
Capping only decreases the objective and leaves its old maximum unchanged.
Identity~\eqref{eq:squares} proves~\eqref{eq:Bvalue}.

Since $W-Q_j=3\mu\norm y^2/2-4H_0+\Gamma_j$, it is at least $H_0$ whenever
$\norm y\ge R_y$. In the transition strip $0<W-Q_j<H_0$, one has
$\norm y\le R_y$. Hence
\[
 \norm{\nabla_y(W-Q_j)}\le 3\mu R_y+12\sqrt{\mu\lambda T}
                         <179\sqrt{\mu\lambda T}.
\]
Formula~\eqref{eq:maxhess} gives, respectively, upper bounds
\[
 768000\lambda+\frac{7680^2\lambda^2T}{768T\lambda},\qquad
 114\sqrt{\mu\lambda}+\frac{7680\cdot179\lambda T\sqrt{\mu\lambda}}{768T\lambda},
 \qquad 4\mu+\frac{179^2\mu\lambda T}{768T\lambda}
\]
on the $uu$, $uy$, and $yy$ cap blocks. Including $D_T$ and the additional dual
quadratic yields~\eqref{eq:Bblock}; the gradient estimate follows because the
$u$ gradient of the cap is $-(1-b)(\Gamma_j)_u$.
\end{proof}

\subsection{Primal protection and scaling}
Let $U\in\St(d_x,T)$ be independent of $V$. Define
\begin{gather*}
 \mathcal R_{R_x}(x)=\frac{x}{\sqrt{1+\norm x^2/R_x^2}},\qquad
 u(x)=\frac{U^\top\mathcal R_{R_x}(x)}r,\\
 \eta_0=c_0/100,\qquad C_x=1000\cdot10^6/c_0,\qquad R_x=C_xr\sqrt T.
\end{gather*}
Set
\begin{equation}\label{eq:Bobj}
 f_{U,V}(x,y)=f_{\rm core}(u(x),y)+\frac{\lambda\eta_0}{2r^2}\norm x^2.
\end{equation}
Its exact primal function is
\begin{equation}\label{eq:Bprimal}
 \Phi_{U,V}(x)=\lambda F_T(u(x))+\frac{\lambda\eta_0}{2r^2}\norm x^2.
\end{equation}

\begin{lemma}[Global certificate and stationarity barrier]\label{lem:Bbarrier}
Take $K_*=10^8$ and
\begin{equation}\label{eq:Bscaling}
 r=\frac{16K_*\epsilon}{c_0L},\qquad
 \lambda=\frac{64K_*\epsilon^2}{c_0^2L},\qquad
 \kappa\ge4K_*,\qquad T=\left\lfloor\frac{\Delta}{32\lambda}\right\rfloor\ge2.
\end{equation}
Then~\eqref{eq:Bobj} is globally $L$-smooth, $\mu$-strongly concave in $y$, has
$G_0=0$, and its primal gap is at most $\Delta$. Whenever $|u_T(x)|\le1/4$,
\begin{equation}\label{eq:Bbarrier}
 \norm{\nabla\Phi_{U,V}(x)}\ge2\epsilon.
\end{equation}
\end{lemma}
\begin{proof}
The compression has Jacobian norm at most $1$ and second derivative operator norm
at most $3/R_x$. Its composition with~\eqref{eq:Bblock} shows that the joint
smoothness constant is at most $K_*(\lambda/r^2+\mu)$.
For clarity, the $xx$ block is bounded by
$10^6\lambda/r^2+24000\lambda\sqrt T/(rR_x)+\lambda\eta_0/r^2$;
the mixed block is at most $2000\sqrt{\mu\lambda}/r$ and the dual block at most
$50\mu$. The inequality $2\sqrt{st}\le s+t$ proves the asserted $K_*$ bound.
As $\lambda/r^2=L/(4K_*)$ and $\mu\le L/(4K_*)$, it is less than $L$.
Strong concavity is unchanged. At $x=y=0$ all square links and their first
partials vanish, the cap is inactive, and $\nabla_y f=0$.
The nonnegative ridge and $|F_T|\le9T$ give gap at most $18T\lambda\le\Delta$.

Here is a direct proof of the barrier, including large outputs. Write
$w=x/r$, $P=C_x\sqrt T$, and $a=(1+\norm w^2/P^2)^{-1/2}$, so $u=aU^\top w$.
The gradient of the normalized primal is
\[
 aU\nabla F_T(u)-\frac{a^2w}{P^2}\ip{u}{\nabla F_T(u)}+\eta_0w.
\]
If $\norm w\le P/2$, choose $j$ from Lemma~\ref{lem:corebounds}. The projection
of this gradient on the $j$-th column of $U$ is at most
\[
 -\frac{2c_0}{\sqrt5}
 +\frac{10^6}{\sqrt5 C_x}+\frac{\eta_0\sqrt5}{2}<-c_0/2.
\]
We used $|u_j|\le1$, $\norm u\le P/\sqrt5$, and
$\norm{\nabla F_T}\le10^6\sqrt T$.
If $\norm w>P/2$, its norm is at least
\[
 \eta_0\norm w-10^6\sqrt T
 > (\eta_0C_x/2-10^6)\sqrt T=4\cdot10^6\sqrt T>c_0/2.
\]
Multiplying by $\lambda/r=4\epsilon/c_0$ proves~\eqref{eq:Bbarrier}.
\end{proof}

\subsection{A fixed globally unbiased probability oracle}
Define the truncated seed, for $1\le j\le T$, by
\[
 D_j(u)=-\Theta(u_1)-C_\Theta\sum_{i=1}^{\min\{j,T-1\}}\Psi(u_i).
\]
Let $f_j(x,y)$ denote~\eqref{eq:Bobj} with $D_j$ and $K_j$ in place of $D_T,K_T$.
In particular $f_T=f_{U,V}$. The inclusion of the $i=j$ seed term is intentional.
For every query $z=(x,y)$ define the \emph{static} index
\begin{equation}\label{eq:kquery}
 k(z)=\max\left(\{1\}\cup\{i:|u_i(x)|>1/4\}
 \cup\{i:|\ip{v_{i,\nu}}{y}|>r\text{ for some }\nu\}\right)\le T.
\end{equation}
No transcript or previously revealed support occurs in this definition.

\begin{lemma}[Uniform residual bound]\label{lem:Bresidual}
For every query, including arbitrary query norms,
\[
 R(z):=\nabla f_T(z)-\nabla f_{k(z)}(z)
 \quad\hbox{satisfies}\quad \norm{R(z)}\le K_Ra_s,\qquad K_R=10^4.
\]
Here $\nabla f_{k(z)}(z)$ means the ordinary gradient of the function with that
fixed integer subscript, not differentiation through the index map.
\end{lemma}
\begin{proof}
For $k=T$ the residual is zero. Otherwise all primal coordinates beyond $k$ have
absolute value at most $1/4$. Their gates and all gate derivatives vanish.
Thus $D_T,D_k$ and their gradients agree, and $\Gamma_T-\Gamma_k$ consists only
of the two links with index $k+1$.  Both of their dual projections have magnitude
at most $r$, so~\eqref{eq:clip} gives
\[
 |\Delta\Gamma|\le8a_sr,\quad
 \norm{\nabla_u\Delta\Gamma}\le60a_sr,\quad
 \norm{\nabla_y\Delta\Gamma}\le8a_s.
\]
The two cap weights differ by at most
$|\Delta\Gamma|/H_0\le a_sr/(96T\lambda)$.
For $\norm y\ge R_y$ both caps are exactly the same radial function, so their
residual is zero. For $\norm y<R_y$, using the gradient estimates from
Lemma~\ref{lem:Bcap} gives
\[
 \norm{R_u}\le60a_sr+80a_sr/\sqrt T\le140a_sr,
\]
\[
 \norm{R_y}\le8a_s+\frac{179}{96}\frac{a_sr\sqrt\mu}{\sqrt{\lambda T}}
             \le10a_s.
\]
The last inequality follows from $r\sqrt\mu/\sqrt\lambda\le1$, which is
part of~\eqref{eq:Bscaling}. Primal composition has Jacobian norm at most $1/r$,
so the original-coordinate residual has norm at most $150a_s<K_Ra_s$.
\end{proof}

For $\sigma>0$ choose
\begin{equation}\label{eq:Bp}
 p=\min\{1,K_R^2a_s^2/\sigma^2\},
 \qquad \zeta\sim\operatorname{Bernoulli}(p),
\end{equation}
and return the fresh vector
\begin{equation}\label{eq:Boracle}
 \mathsf G(z)=\nabla f_{k(z)}(z)+\frac\zeta p
                    [\nabla f_T(z)-\nabla f_{k(z)}(z)].
\end{equation}
For $\sigma=0$ define $p=1$ and use the same formula. This is a fixed measurable
Markov kernel for every sampled $U,V$. Exactly,
\[
 \E[\mathsf G(z)]=\nabla f_T(z),\qquad
 \E\norm{\mathsf G(z)-\nabla f_T(z)}^2
  =(1/p-1)\norm{R(z)}^2\le\sigma^2.
\]
The Bernoulli probability is a constant independent of the query and of both
hidden frames. No claim is made that a realized oracle sample is itself a gradient
of a sample function; the problem's oracle contract does not require that property.

\subsection{Product-space lifting, with the final output included}
\begin{theorem}[Unrestricted stochastic progress]\label{thm:B}
For the parameters above and an integer horizon
$N\le(T-1)/(12p)$, choose
\begin{align}
 d_x&=T+3+\left\lceil64\frac{R_x^2}{r^2}
                       \log(96T(N+1))\right\rceil,\nonumber\\
 d_y&=2(T-1)+3+\left\lceil64\frac{R_y^2}{r^2}
                       \log(192T(N+1))\right\rceil.\label{eq:Bdimensions}
\end{align}
On the fixed distribution of independent uniform $U,V$ and the fixed oracle
\eqref{eq:Boracle}, every unrestricted adaptive randomized algorithm has success
probability at most $1/6$ after $N$ calls.
Consequently, in the stated nontrivial parameter regime,
\[
 \Tloc\ge c\left(\frac{L\Delta}{\epsilon^2}
              +\frac{\kappa L\Delta\sigma^2}{\epsilon^4}\right),
 \qquad G_0=0.
\]
\end{theorem}
\begin{proof}
In a more informative auxiliary interaction, reveal $U$'s first column at the
start. Let $K=1$ be the number of revealed primal stages. The revealed dual columns
are exactly those of links $2,\ldots,K$. Draw a fresh independent coin~\eqref{eq:Bp}
at each call. On a success reveal primal column $K+1$ and its two dual link columns,
and increase $K$ by one, until $K=T$. Give the algorithm these columns and the coins
for free. This additional information can only strengthen it.

For $\norm y<R_y$, define $k_0$ by~\eqref{eq:kquery} using only the columns
revealed \emph{before} the query. On a failure, the auxiliary reply is
$\nabla f_{k_0}(z)$. On a success, it is
\[
 \nabla f_{k_0}(z)+\frac1p
       [\nabla f_{\min\{K+1,T\}}(z)-\nabla f_{k_0}(z)],
\]
where here $K$ is the exposure count before that success. All quantities are
computable using the columns now exposed.  For $\norm y\ge R_y$, the auxiliary
reply uses the exposed truncated seed and the common fully radial cap; the
residual is zero. Its coins still expose additional columns, which only makes the
auxiliary interaction more informative. These rules define the auxiliary process
on \emph{every} trajectory, not only on a good event.

The good projections before each query are
\[
 |\ip{u_j}{\mathcal R_{R_x}(x)}|\le r/4\quad(j>K),
\]
\[
 |\ip{v_{i,\nu}}{y}|\le r\quad(i>K,\ \norm y<R_y).
\]
On this event $k(z)=k_0\le K$ for a small dual query. Gates after that stage are
flat. A non-reveal reply therefore uses only the known columns; a reveal reply
can involve at most the next stage. This proves exact equality with the auxiliary
reply. For a large dual query, all caps are the same radial function, regardless
of how large its hidden dual projections are. The primal good projections ensure
that all hidden seed terms are flat. Equality holds in this case as well.

Append the final primal output as a dummy query, imposing the same compressed
primal projection bounds on its unrevealed columns. No oracle response is needed
for this dummy query. Under the auxiliary filtration, Lemma~\ref{lem:ghost}
leaves independent uniform completions of $U,V$. The informative primal query has
norm at most $R_x$ at every step and at the output; a dual threshold is tested only
at norm at most $R_y$. Lemma~\ref{lem:sphere} and~\eqref{eq:Bdimensions}, followed
by a union bound over at most $T(N+1)$ primal and $2TN$ dual projections, give a
total violation probability less than $1/12$.

The number of stages exposed is at most $1+\sum_{t=1}^N\zeta_t$. Hence
\[
 \Pp(K_N=T)\le\Pp\!\left(\sum_{t=1}^N\zeta_t\ge T-1\right)
            \le\frac{Np}{T-1}\le\frac1{12}.
\]
If no violation occurs and $K_N<T$, the output has $|u_T(\widehat x)|\le1/4$,
so Lemma~\ref{lem:Bbarrier} excludes success. The union of the two exceptional
events has probability at most $1/6$. This includes the algorithm's seed and
requires no restriction on its measurable query maps.

Finally $a_s^2=4\mu\lambda$ is a fixed universal constant times
$\epsilon^2/\kappa$. Therefore $1/p$ is within universal constant factors of
$1+\kappa\sigma^2/\epsilon^2$. Under a sufficiently large universal gap
threshold, $T=\Theta(L\Delta/\epsilon^2)$. Integer rounding changes only the
universal lower constant; if the displayed excluded horizon is below one, the
disjoint zero-query testing family below supplies the corresponding one-call
bound. This proves the rate.
\end{proof}

\section{A Two-Dimensional Unrestricted Testing Family}\label{sec:C}

A single primal coordinate and a single dual coordinate suffice for the local statistical lower bound. Use the clip in
\eqref{eq:clip}, and set $\psi_R(y)=\rho_R(y)^2/2$. Then
\[
 |\psi_R'(y)|\le3R/2,\qquad |\psi_R''(y)|\le13/4.
\]
For $\kappa\ge1120$ set $a=b=L/16$ and
\begin{align*}
 Q_+&=a+b^2/(5\mu),& Q_-&=a+b^2/(7\mu),\\
 d&=2\epsilon\frac{Q_++Q_-}{Q_+-Q_-},& R&=\frac{2(d+\epsilon)}b.
\end{align*}
One has $12\epsilon\le d\le13\epsilon$ and $R\le448\epsilon/L$.
Define the fixed pair
\begin{equation}\label{eq:Cobj}
 f_s(x,y)=\frac a2x^2+dx+bxy-3\mu y^2+s\mu\psi_R(y),\qquad s\in\{-1,1\}.
\end{equation}
Use exact primal derivatives and independent $N(0,\sigma^2)$ noise in the dual
component at every query.

\begin{theorem}[True-primal-gradient testing]\label{thm:C}
For $\sigma>0$, the family~\eqref{eq:Cobj} is globally admissible, has $G_0=0$,
and has primal gap at most $151424\epsilon^2/(\kappa L)$. Any arbitrary adaptive
randomized algorithm succeeding with probability $2/3$ on both signs needs
\[
 N\ge\frac{2}{9\cdot903168}\frac{\kappa^2\sigma^2}{\epsilon^2}.
\]
A zero-query algorithm, deterministic or randomized, cannot succeed on both signs with probability
$2/3$, even when $\sigma=0$.
\end{theorem}
\begin{proof}
The dual second derivative is at most $-6\mu+(13/4)\mu<-\mu$ in the sense of
quadratic forms. The joint Hessian norm is bounded by
$\max\{a,(37/4)\mu\}+b<L$. At $(0,0)$ its dual derivative is zero.
For $x$ in
\[
 I_s=\left[\frac{-d-\epsilon}{Q_s},\frac{-d+\epsilon}{Q_s}\right],
\]
the quadratic dual optimizer $y_s=bx/((6-s)\mu)$ satisfies
$|y_s|\le(d+\epsilon)/b=R/2$. It is therefore also the unique global dual
optimizer of the clipped objective. In this region $\Phi_s'(x)=Q_sx+d$.
Global $a$-strong convexity of $\Phi_s$ makes $I_s$ its exact successful set.
The distance between the centers of $I_+,I_-$ is twice the sum of their radii,
by the definition of $d$, so the sets are disjoint. The optimum itself is in the
same region, giving
\[
 \Phi_s(0)-\inf\Phi_s=\frac{d^2}{2Q_s}
 \le\frac{896\cdot13^2\epsilon^2}{\kappa L}.
\]

At every arbitrary query, the exact primal derivative $ax+d+by$ is identical
under both signs, while the dual mean difference is $2\mu\psi_R'(y)$.
The conditional KL divergence of the oracle replies is at most
\[
 \frac{(3\mu R)^2}{2\sigma^2}
 \le903168\frac{\epsilon^2}{\kappa^2\sigma^2}.
\]
Include the independent algorithm seed in the transcript. The chain rule for
relative entropy sums this bound over the $N$ adaptive queries, since at a fixed
past transcript the next query is the same measurable map under both hypotheses.
The final output is a measurable function of the transcript, so its total
variation distance is no larger. Success on both signs implies that this distance
is at least $1/3$, using the event $\widehat x\in I_+$. Pinsker's inequality
$\TV^2\le\KL/2$ yields the stated lower bound. A zero-query output has the same
law under both signs; disjointness alone proves the last assertion.
\end{proof}

\section{The Accelerated Exact-Oracle Lower Bound}\label{sec:Anew}

We now give the missing exact-oracle construction. Its starting point is the
interleaved nonlinear/primal and quadratic/dual chain of Li et al.~\cite{Li2021}.
The changes below are essential: Robin endpoint penalties make every dual
interaction convex before negation, bounded-slope edges allow primal compression
without a condition-number loss, and flat edges give exact transcript locality.

\subsection{A bounded-coefficient Robin chain}
Fix integers $n\ge10$ and $T\ge2$. Let $A_n$ be the path Laplacian, with
endpoint diagonal entries $1$, other diagonal entries $2$, and off-diagonal
entries $-1$. Put
\[
 H_n=n^{-2}I+A_n+n^{-1}(e_1e_1^\top+e_ne_n^\top),\quad
 a_1=(H_n^{-1})_{11}/n,\quad a_2=(H_n^{-1})_{1n}/n,
\]
\[
 C=12/a_2,\qquad c=C(1+a_2-a_1)/2.
\]
These coefficients obey
\begin{equation}\label{eq:Robinbounds}
 1/7\le a_2\le a_1<1,\qquad 12\le C\le84,\qquad 0<c\le42.
\end{equation}
Here is a direct verification. Set $\cosh\theta=1+1/(2n^2)$ and
$h=(n^{-1}+(2n^2)^{-1})/\sinh\theta$. Then $0<\theta\le1/n$ and
$1\le h\le1.05$. The solution of the left homogeneous boundary recurrence is
\[
 u_i=\cosh((i-1)\theta)+h\sinh((i-1)\theta),\qquad u_1=1.
\]
It has $u_i\le3$. The right boundary denominator is
\[
 d_n=(n^{-1}+(2n^2)^{-1})u_n+
 \sinh\theta\,[\sinh((n-1)\theta)+h\cosh((n-1)\theta)]\le7/n.
\]
Thus $(H_n^{-1})_{1n}=1/d_n\ge n/7$.
The inverse is entrywise positive, by the discrete maximum principle.
The same principle bounds its endpoint-source columns by their endpoint maxima.
Finally the variational formula for the inverse and
$H_n\succeq n^{-1}(e_1e_1^\top+e_ne_n^\top)$ imply that its two-by-two
endpoint inverse block is strictly less than $nI$. Consequently
$0<a_1-a_2<1$, proving all the other assertions in~\eqref{eq:Robinbounds}.

Use the clip $\rho_2$ in~\eqref{eq:clip}, and write
$u=(x_1,z_2,x_2,\ldots,z_T,x_T)\in\R^{2T-1}$, with $q=\rho_2(u)$
coordinatewise. In particular $|q_j|\le3$. With the $\Psi,\Theta$ already defined
in Section~\ref{sec:B}, define
\[
 N(q)=-\Theta(q_{x_1})+\sum_{i=2}^T
 [\Psi(-q_{z_i})\Theta(-q_{x_i})-\Psi(q_{z_i})\Theta(q_{x_i})].
\]
For each $i<T$, let $y^{(i)}\in\R^n$. Introduce the $n+1$ edge arguments
\begin{align*}
 t_{i,0}&=y^{(i)}_1/\sqrt n-\sqrt C\,q_{x_i},\\
 t_{i,j}&=y^{(i)}_j-y^{(i)}_{j+1},\quad 1\le j<n,\\
 t_{i,n}&=y^{(i)}_n/\sqrt n+\tfrac{\sqrt C}{2}q_{z_{i+1}}.
\end{align*}
If all edges have energy $t^2/2$, elimination of $y$ gives exactly
\begin{equation}\label{eq:Aclippedideal}
 J(u)=N(q)+6\sum_{i<T}(q_{x_i}-q_{z_{i+1}}/2)^2
\end{equation}
when the primal-only quadratic is $c\sum_{i<T}(q_{x_i}^2+q_{z_{i+1}}^2/4)$.
Indeed the expanded dual Hessian is $H_n$, its source is
$\sqrt{C/n}(q_{x_i}e_1-q_{z_{i+1}}e_n/2)$, and substitution of its inverse
leaves the cross coefficient $-Ca_2/2=-6$ and diagonal coefficients $6,3/2$.

\subsection{Flat bounded-slope edges and dual protection}
For $b=1000$ and $\delta>0$, define the convex, continuously differentiable function
\[
 \varphi_{\delta,b}(t)=
 \begin{cases}
 0,&|t|\le\delta,\\
 (|t|-\delta)^2/2,&\delta<|t|\le\delta+b,\\
 b(|t|-\delta)-b^2/2,&|t|>\delta+b.
 \end{cases}
\]
Its derivative is $1$-Lipschitz, has magnitude at most $b$, and differs from
$\varphi_{0,b}'$ by at most $\delta$ everywhere. Keep $\varphi_\delta$ from
Section~\ref{sec:D} for the primal-only terms. Put $M=(n+1)(T-1)$ and
\begin{gather*}
 Q_\delta(u,y)=\sum_{i<T}\sum_{j=0}^n\varphi_{\delta,b}(t_{i,j}),\qquad
 D_\delta(u)=N(q)+2c\sum_{i<T}[\varphi_\delta(q_{x_i})+
                                    \varphi_\delta(q_{z_{i+1}})/4],\\
 R_y=1000n\sqrt T,\quad W(y)=20(\norm y^2-R_y^2),\quad
 \overline Q_\delta=\mathcal M_{R_y^2}(Q_\delta,W),\\
 F_\delta(u,y)=D_\delta(u)-\frac{\norm y^2}{2n^2}-\overline Q_\delta(u,y).
\end{gather*}
In ambient dual space, the edge arguments use $V^\top y$ for a frame with
$n(T-1)$ columns; both norms above remain the full ambient norm.
Set
\begin{equation}\label{eq:Adelta}
 c_0=e^{-1/2},\quad \eta=c_0/100,\quad
 \delta=\frac{\eta}{10^8 n^2\sqrt M},\quad K=10^8.
\end{equation}

\begin{lemma}[Global and value-function certificates]\label{lem:Acore}
$F_\delta$ is $n^{-2}$-strongly concave in $y$. Its joint gradient is
$K$-Lipschitz and $\norm{\nabla_u F_\delta}\le K\sqrt T$ globally.
Its maximum $J_\delta(u)=\max_yF_\delta(u,y)$ satisfies
\begin{equation}\label{eq:Avalueerror}
 \norm{\nabla J_\delta(u)-\nabla J(u)}\le\eta/100
 \quad\hbox{for every }u.
\end{equation}
Moreover, $J_\delta(u)\ge-1000T$, $J_\delta(0)=-\Theta(0)$, and its dual
maximizer at $u=0$ is zero. For $\norm y\ge2R_y$, the cap is exactly
$W-R_y^2/2$, independent of $u$ and of the entire dual frame.
\end{lemma}
\begin{proof}
All edge energies are convex in $y$, as is the smooth maximum, so strong
concavity follows from the separate isotropic quadratic.
Write the vector of edge arguments as $A_yV^\top y+A_xq$.
Its matrices satisfy $\norm{A_y}^2\le5$ and $\norm{A_x}^2\le84$.
Each primal coordinate occurs in just one edge. The derivative bounds for the
clip and the bounded edge slope therefore give, uniformly over $y$,
\[
 \norm{(Q_\delta)_u}\le14000\sqrt T,\quad
 \norm{(Q_\delta)_{uu}}\le10^4,\quad
 \norm{(Q_\delta)_{uy}}\le21,\quad
 0\preceq(Q_\delta)_{yy}\preceq5I.
\]
Also $0\le Q_\delta\le2000T+5\norm y^2$. The cap equals $Q_\delta$ on
$\norm y\le R_y$, and equals $W-R_y^2/2$ on $\norm y\ge2R_y$.
Its transition strip has $\norm y<2R_y$; there
$\norm{\nabla_y(W-Q_\delta)}\le100R_y$.
Formula~\eqref{eq:maxhess} bounds the added $uu,uy,yy$ blocks by
$196/n^2$, $1400/n$, and $10000$, respectively. Adding $D_\delta$, whose
component derivative bounds follow directly from the fixed scalar functions,
and the isotropic quadratic proves the asserted loose constant $K$.
It also proves the bound on the primal partial gradient.

For the value comparison, first take $\delta=0$ without the radial cap.
The unmodified quadratic optimizer of one dual subchain is
\[
 y^*=\sqrt{C/n}\,H_n^{-1}(q_xe_1-q_ze_n/2).
\]
The endpoint inverse bound and the maximum principle imply
$\max_j|y_j^*|\le42\sqrt n$ and hence $\norm{y^*_{\rm all}}\le42n\sqrt T$.
The boundary equation and the interior recurrence show that
$|y_j^*-y_{j+1}^*|\le200/\sqrt n$; explicitly the first difference is at most
$[3\sqrt C+42(1+1/n)]/\sqrt n$, and all the subsequent increments together
are at most $42/\sqrt n$.
The two endpoint edge arguments have magnitude at most $70$.
Thus every edge is in the quadratic part of $\varphi_{0,b}$ at this optimizer.
Strong convexity in the dual minimization problem proves that it is still the
unique global optimizer for the Huberized problem. Its value is~\eqref{eq:Aclippedideal}.

Replacing $\varphi_{0,b}$ by $\varphi_{\delta,b}$ perturbs the dual gradient
by at most $3\delta\sqrt M$. Strong monotonicity moves its optimizer by at most
$3n^2\delta\sqrt M$. The direct primal gradient perturbation, including the
primal-only flattened quadratics, is at most $100\delta\sqrt M$.
The mixed Lipschitz bound above therefore gives
\[
 \norm{\nabla J_\delta-\nabla J}\le200n^2\delta\sqrt M<\eta/100
\]
for the uncapped problem. Both relevant optimizers have norm less than $R_y/2$.
The cap only decreases the objective and is inactive there, so their values
and primal derivatives are unchanged by capping.
At $y=0$, the positive primal-only terms and the bounded nonlinear terms give
$F_\delta(u,0)\ge-1000T$, proving the value lower bound.
At $u=y=0$, all edge derivatives and all flattened quadratic terms vanish.
The surviving seed has value $-\Theta(0)$ and zero dual derivative.
Strong concavity therefore gives the stated maximizer and value at the origin.
\end{proof}

\subsection{Primal protection and the stationarity barrier}
Let $P=2T-1$, let $U\in\St(d_x,P)$ be independent of $V$, and set
\[
 R_x=C_x\sqrt T,\qquad C_x=10^4K/\eta,\qquad
 u(w)=U^\top\mathcal R_{R_x}(w).
\]
Define the normalized objective
\begin{equation}\label{eq:Anormalized}
 \mathfrak f_{U,V}(w,y)=F_\delta(u(w),y)+\frac\eta2\norm w^2.
\end{equation}
It is $L_A$-smooth for $L_A=10^{10}$, is $n^{-2}$-strongly concave in $y$,
and its initial primal gap is at most $1000T$. Indeed the compression Jacobian
has norm at most one and its second derivative has norm at most $3/R_x$;
Lemma~\ref{lem:Acore} bounds the extra Hessian term by $3K/C_x$.

\begin{lemma}[Robust end game]\label{lem:Abarrier}
If the compressed coordinate $|u_{z_T}(w)|\le1/4$, then
\[
 \norm{\nabla\max_y\mathfrak f_{U,V}(w,y)}\ge\eta/2.
\]
\end{lemma}
\begin{proof}
Write $a=(1+\norm w^2/R_x^2)^{-1/2}$. The primal gradient is
\[
 aU\nabla J_\delta(u)-\frac{a^2w}{R_x^2}
                    \ip{u}{\nabla J_\delta(u)}+\eta w.
\]
For $\norm w>R_x/2$ its norm is at least
$\eta\norm w-K\sqrt T>\eta/2$.
Otherwise $a\ge2/\sqrt5$. Its projection onto any column of $U$ is
\[
 a\partial_j J_\delta(u)+b_0u_j,\qquad
 b_0=\eta/a-a\ip{u}{\nabla J_\delta(u)}/R_x^2,
\]
where $0.99\eta\le b_0\le1.13\eta$.
Let $j\ge2$ be the first index with $|u_{z_j}|<1$, with $z_1=1$ by convention.
Then $|q_{z_{j-1}}|\ge1$ (or this is the deterministic seed).
Put $v=u_{x_{j-1}}$ and $z=u_{z_j}$. In the ideal clipped function, the two
relevant derivatives are
\[
 \partial_vJ=\rho_2'(v)[p+12(\rho_2(v)-z/2)],\qquad
 \partial_zJ=q'-6(\rho_2(v)-z/2),
\]
where $-3<p\le0$, $q'\le0$, and $p\le-c_0$ when $|v|<1$.
If $|v|<1$, the combination $\partial_vJ+2\partial_zJ\le-c_0$.
The corresponding combination of actual gradient projections is at most
$-ac_0+3b_0+a\sqrt5\,\eta/100<-c_0/2$.
If $v\ge1$, the bracket in the first derivative is at least $3$; if $v\le-1$
it is at most $-6$. Since $\rho_2'\ge0$, this derivative cannot oppose
the ridge contribution $b_0v$. Consequently the magnitude of its actual projection is at least
$b_0|v|-a\eta/100>\eta/2$.
This covers every case, including the clip transition and saturation regions.
\end{proof}

\subsection{Exact adaptive lifting and parameter restoration}
In physical coordinates define
\[
 f(x,y)=\lambda\mathfrak f_{U,V}(x/r,y/r),\qquad
 r=\frac{8L_A\epsilon}{\eta L},\quad
 \lambda=\frac{32L_A\epsilon^2}{\eta^2L},\quad
 n=\left\lfloor\sqrt{\frac\kappa{4L_A}}\right\rfloor,\quad
 T=\left\lfloor\frac\Delta{1000\lambda}\right\rfloor.
\]
For sufficiently large universal $\kappa_0,C_0$, these choices have $n\ge10,T\ge2$.
Joint smoothness is at most $L/2$, dual strong concavity is at least $2\mu$,
the gap is at most $\Delta$, and $G_0=0$.
Set $Q=n(T-1)$, $S=(n+2)(T-1)$, and
\[
 \tau=\delta/[8(1+\sqrt C)].
\]
For a horizon $N\le(S-2)/2$, explicit finite dimensions are
\begin{align}
 d_x&=P+3+\left\lceil\frac{2R_x^2}{\tau^2}
                \log(48(P+Q)(N+1))\right\rceil,\nonumber\\
 d_y&=Q+3+\left\lceil\frac{8R_y^2}{\tau^2}
                \log(48(P+Q)(N+1))\right\rceil.\label{eq:Adim}
\end{align}

\begin{theorem}[Accelerated unrestricted exact-oracle lower bound]\label{thm:Anew}
For universal $c,C_0,\kappa_0>0$, if $\kappa\ge\kappa_0$ and
$L\Delta\ge C_0\epsilon^2$, then
\[
 \Tloc(L,\mu,\Delta,0,0,\epsilon)
 \ge c\frac{\sqrt\kappa L\Delta}{\epsilon^2}.
\]
The same lower bound holds for every $G,\sigma\ge0$.
\end{theorem}
\begin{proof}
Use exact joint gradients and sample $U,V$ independently and uniformly before
interaction. The vertex order is
\[
 x_1,y^{(1)}_1,\ldots,y^{(1)}_n,z_2,x_2,
 y^{(2)}_1,\ldots,y^{(2)}_n,z_3,x_3,\ldots,z_T,x_T.
\]
An auxiliary interaction gives the first primal column for free, exposes one
additional vertex at each call, and evaluates the core with all still unexposed
coordinates set to zero. Compression derivatives and all radial terms use the
full query vectors. This defines a process on every trajectory.
At a small dual query $\norm{y/r}\le2R_y$, impose that every pre-query unexposed
primal compressed projection and dual projection has absolute value at most $\tau$.
Every edge with two such endpoints is exactly flat, as is every hidden
primal-only quadratic. Every nonlinear outgoing gate at a hidden $z_i$ is also
flat. After exposing the next vertex, all boundary edges are computed exactly;
edges from that newly exposed vertex to still hidden vertices remain flat.
Thus the auxiliary and true values and gradients agree exactly.
At a large dual query the cap is fully radial, and the same primal good event
suffices. All informative primal vectors have norm at most $R_x$.

Append the final output as a dummy primal query. Lemmas~\ref{lem:ghost}
and~\ref{lem:sphere}, with~\eqref{eq:Adim}, bound the probability of any
projection violation by $1/12$ under the unstopped auxiliary interaction.
At most $1+N<S$ vertices have been exposed; in particular $z_T$ remains hidden.
On the coupled good event its output projection is at most $\tau<1/4$.
Lemma~\ref{lem:Abarrier} and $\lambda/r=4\epsilon/\eta$ then give a true
primal gradient of magnitude at least $2\epsilon$.
Hence average success is at most $1/12$, even for an arbitrary independent
algorithm seed and arbitrary measurable, unbounded queries.
Finally $S=\Theta(\sqrt\kappa L\Delta/\epsilon^2)$. Rounding is absorbed by
increasing the universal nontrivial-regime constants. The exact oracle and
$G_0=0$ place this family in every larger noise/localization budget.
\end{proof}

\section{A Fresh-Gradient Solver with Additive Initialization Cost}\label{sec:uppernew}

This section constructs and analyzes the fresh-gradient SCSC solver used by the outer method. All sample
complexities in the intermediate lemmas are expectations of the number of
\emph{original} one-point calls. Section~\ref{sec:outernew} converts the entire
algorithm, rather than each inner call separately, to a fixed query budget.
No samplewise smoothness, sample reuse, or gradient-difference interface is used.

\subsection{A finite multilevel oracle with relative bias}
Let $g(u,v)$ be jointly $\ell$-smooth, $\alpha$-strongly convex in $u$, and
$\mu$-strongly concave in $v$, with $\alpha\ge\ell/4$. Write
\[
 u^*(v)=\arg\min_u g(u,v),\qquad q(v)=-\min_u g(u,v).
\]
Strong monotonicity in $u$ and the two first-order optimality equations imply
$\alpha\norm{u^*(v)-u^*(v')}\le\ell\norm{v-v'}$.
Thus $u^*$ is $k$-Lipschitz for $k\le4$. The function $q$ is $\mu$-strongly
convex as the pointwise maximum of the functions $-g(u,\cdot)$.
Danskin's identity gives $\nabla q(v)=-\nabla_v g(u^*(v),v)$, whose Lipschitz
constant is at most $\ell+\ell k\le H_0=5\ell$. We will also use
$Q(v)=q(v)+r(v)$, where $r$ is a known convex quadratic. Suppose $Q$ is
$m$-strongly convex and $H$-smooth, with
\begin{equation}\label{eq:valuerange}
 0<m\le H\le20\ell.
\end{equation}
The added quadratic is evaluated exactly; it does not change $u^*(v)$.
Let $h=17$ and fix $\rho=(10^{10}h^2)^{-1}$ throughout this section.

\begin{lemma}[Relative-error effective gradient]\label{lem:reloracle}
Fix a reference $\bar u$, a target value accuracy $e>0$, and parameters $m,H$ as
above. There is a finite fresh-sample estimator $Y(v)$ of $\nabla Q(v)$ with
\begin{align}
 \norm{\E[Y(v)]-\nabla Q(v)}^2
 &\le\rho m e+\rho m^2\norm{\bar u-u^*(v)}^2,\label{eq:relbias}\\
 \E\norm{Y(v)-\E Y(v)}^2
 &\le V+\rho m^2\norm{\bar u-u^*(v)}^2,\label{eq:relvariance}
\end{align}
where $V\le C(K+1)\sigma^2$ and
$K=O(1+\log(2+\sigma^2/(me)))$. Its expected number of original calls is at most
\begin{equation}\label{eq:levelcost}
 C(K+1)\big[1+K+\log(K+1)+\log(2+\ell/m)\big].
\end{equation}
These assertions hold conditionally on arbitrary previous information fixing
$g,r,\bar u,v$. Neither $K$ nor the cost bound depends on the distance of $\bar u$
from $u^*(v)$. At $\sigma=0$ the estimator is deterministic and has $K=0$.
\end{lemma}
\begin{proof}
For a level $j$, run batch gradient descent on $g(\cdot,v)$ from $\bar u$ with
step $1/\ell$, using $2^j$ fresh samples at each step. If $U$ is the current
point, conditional unbiasedness and the contraction of the exact gradient step
imply
\[
 \E[\norm{U^+-u^*(v)}^2\mid U]
 \le(1-\alpha/\ell)\norm{U-u^*(v)}^2+\frac{\sigma^2}{\ell^2 2^j}.
\]
For clarity, the contraction follows by integrating the Hessian along the
segment to $u^*(v)$: the averaged Hessian has eigenvalues in $[\alpha,\ell]$,
so the squared contraction is at most $(1-\alpha/\ell)^2\le1-\alpha/\ell$.
The same assertion follows by smooth approximation when only a Lipschitz gradient
is assumed.

After $s_j$ steps, form $H_j(v)$ by taking a further independent average of $2^j$
negative dual gradient samples at $(U,v)$ and adding $\nabla r(v)$. Take $C_v=10$. For any $\gamma\in(0,1)$, choosing
\[
 s_j=\left\lceil\frac\ell\alpha
 \log\left[2^{j+2}\big(1+\ell^2/(\gamma m^2)\big)\right]\right\rceil
\]
gives, with $D=\norm{\bar u-u^*(v)}$,
\begin{equation}\label{eq:levelMSE}
 \E\norm{H_j(v)-\nabla Q(v)}^2
 \le2^{-j}\big(\gamma m^2D^2+C_v\sigma^2\big).
\end{equation}
Indeed the initial-distance contribution is at most
$2\ell^2(1-\alpha/\ell)^{s_j}D^2$, while the limiting mean-square primal error
is at most $\sigma^2/(\alpha\ell 2^j)$; the final batch adds at most
$\sigma^2/2^j$. Since $\ell/\alpha\le4$, the displayed choices give
exactly~\eqref{eq:levelMSE}. The cost is $(s_j+1)2^j$.

When $\sigma>0$, choose
\[
 K=\max\left\{1,\left\lceil\log_2\frac{C_v\sigma^2}{\rho m e}\right\rceil\right\},
 \qquad \gamma=\frac{\rho}{14(K+1)},\qquad
 p_j=\frac{2^{-j}}{1-2^{-K}},\quad 1\le j\le K.
\]
Choose an independent level $J$ with probabilities $p_j$, and put
\[
 Y=H_0+\frac{H_J-H_{J-1}'}{p_J}.
\]
Every appearance of a level estimator uses independent fresh samples, including
$H_0$ and $H_{J-1}'$ when $J=1$. Expectations telescope, so $\E Y=\E H_K$.
Writing $A=\gamma m^2D^2+C_v\sigma^2$, inequality~\eqref{eq:levelMSE} gives
\[
 \E\norm{Y-\nabla Q(v)}^2
 \le2A+2\sum_{j=1}^K\frac{6A2^{-j}}{p_j}
 \le14(K+1)A.
\]
Thus the variance bound holds with $V=14C_v(K+1)\sigma^2$.
The squared bias is at most $2^{-K}A$, proving~\eqref{eq:relbias}.
Summing $p_j$ times the costs of the two sampled levels proves~\eqref{eq:levelcost}.
When $\sigma=0$, take $K=0$, $\gamma=\rho/14$, and $Y=H_0$. Its variance is zero
and~\eqref{eq:levelMSE} proves the bias bound directly. In particular the exact
endpoint uses neither a singular level distribution nor a limiting argument.
\end{proof}

\subsection{An explicit robust accelerated recurrence}
\begin{lemma}[Acceleration with conditional bias and variance]\label{lem:agdbias}
Let $Q$ be $H$-smooth and $m$-strongly convex, and put $a=\sqrt{m/H}$.
Starting from $x_0=z_0$, perform
\begin{align*}
 y_t&=(x_t+a z_t)/(1+a),\\
 x_{t+1}&=y_t-g_t/H,\\
 z_{t+1}&=(1-a)z_t+a y_t-(a/m)g_t.
\end{align*}
Let $v^*=\arg\min Q$ and
$\mathcal P_t=Q(x_t)-Q(v^*)+(m/2)\norm{z_t-v^*}^2$.
If $g_t$ has conditional bias $b_t$ relative to $\nabla Q(y_t)$ and conditional
centered variance $V_t$, then
\begin{equation}\label{eq:AGDrobust}
 \E[\mathcal P_{t+1}\mid\mathcal F_t]
 \le(1-a/2)\mathcal P_t+V_t/H+(2a/m)\norm{b_t}^2.
\end{equation}
\end{lemma}
\begin{proof}
Write $d=\nabla Q(y_t)$ and
$z^0=(1-a)z_t+a y_t-(a/m)d$. Expanding the squared norm and using
$a^2/m=1/H$ yields the exact-gradient inequality
\begin{equation}\label{eq:exactpotential}
 Q(y_t)-Q(v^*)-\frac{\norm d^2}{2H}
       +\frac m2\norm{z^0-v^*}^2\le(1-a)\mathcal P_t.
\end{equation}
Here are the scalar inequalities needed for that expansion. Use
$\langle d,y_t-v^*\rangle\ge Q(y_t)-Q(v^*)+(m/2)\norm{y_t-v^*}^2$
and $\langle d,x_t-y_t\rangle\le Q(x_t)-Q(y_t)$, together with
$x_t-y_t=-a(z_t-y_t)$ and
\[
 \norm{(1-a)(z_t-v^*)+a(y_t-v^*)}^2
 =(1-a)\norm{z_t-v^*}^2+a\norm{y_t-v^*}^2
       -a(1-a)\norm{z_t-y_t}^2.
\]
After cancellation the remaining square has a nonpositive coefficient, proving
\eqref{eq:exactpotential}. Smoothness and $Q(y_t-d/H)\ge Q(v^*)$ show that the
first three terms in~\eqref{eq:exactpotential}, excluding its norm term, are
nonnegative. Consequently $\norm{z^0-v^*}^2\le2\mathcal P_t/m$.

Put $e_t=g_t-d$. Smoothness bounds the noisy primal step by
$Q(y_t)-\norm d^2/(2H)+\norm{e_t}^2/(2H)$.
The change in the dual potential is
$-a\langle z^0-v^*,e_t\rangle+\norm{e_t}^2/(2H)$.
Take conditional expectations and use
\[
 a\norm{b_t}\sqrt{2\mathcal P_t/m}
 \le(a/2)\mathcal P_t+(a/m)\norm{b_t}^2,
 \qquad H^{-1}\le a/m.
\]
Together with~\eqref{eq:exactpotential}, this proves~\eqref{eq:AGDrobust}.
\end{proof}

\subsection{Value-accuracy solves without a radius-dependent noise cost}
\begin{lemma}[Localized value solver]\label{lem:valuesolver}
Under~\eqref{eq:valuerange}, let $v^*=\arg\min Q$ and $u^*=u^*(v^*)$.
Suppose the starting pair $(\bar u,\bar v)$ has a known bound
\begin{equation}\label{eq:valueenergy}
 \E[Q(\bar v)-Q(v^*)+m\norm{\bar u-u^*}^2]\le E_0.
\end{equation}
For every $e>0$, a fresh-gradient method returns $(u,v)$ with
\[
 \E[Q(v)-Q(v^*)]\le e,\qquad
 m\E\norm{u-u^*(v)}^2\le e.
\]
Writing
$\mathcal L=1+\log(2+\ell/m)+\log(2+\sigma^2/(me))$, its expected original-call
cost is bounded by
\begin{equation}\label{eq:valuecost}
 C\mathcal L^3\left[\sqrt{H/m}\log(2+E_0/e)+\frac{\sigma^2}{me}\right].
\end{equation}
The function and the starting points may be random, provided the uniform
regularity assumptions hold and~\eqref{eq:valueenergy} is interpreted in
expectation. No almost-sure starting radius or compact auxiliary domain is required.
\end{lemma}
\begin{proof}
We describe one contracting epoch. Its input envelope is $E\ge e$ and bounds
\eqref{eq:valueenergy}. Freeze the input $\bar u$ during this epoch.
Run Lemma~\ref{lem:agdbias} from $x_0=z_0=\bar v$, using the average of
\[
 b=\max\left\{1,\left\lceil\frac{10^8hV}{E\sqrt{mH}}\right\rceil\right\}
\]
independent estimators from Lemma~\ref{lem:reloracle} at each effective query.
The estimator's level cutoff is selected using the final desired accuracy $e$,
not the epoch envelope or the initial distance.

Let $p_t=\E\mathcal P_t$. Strong convexity gives $p_0\le2E$, and the Lipschitz
bound $k\le4$ gives, conditionally at an effective query,
\[
 m\norm{\bar u-u^*(y_t)}^2
 \le2m\norm{\bar u-u^*}^2+4k^2\mathcal P_t.
\]
Indeed $y_t$ is a convex combination of $x_t,z_t$, whose squared distances to
$v^*$ are each at most $2\mathcal P_t/m$.
Averaging reduces the centered variance by $b$ and does not change the bias.
Substitute~\eqref{eq:relbias}--\eqref{eq:relvariance} into
\eqref{eq:AGDrobust}, and use $a^2=m/H\le a$, $e\le E$.
The coefficient added to $p_t$ is at most $12\rho k^2a$, and the other added terms
are at most $8\rho a E+V/(Hb)$. Our choices therefore give the explicit recurrence
\[
 p_{t+1}\le(1-a/3)p_t+\frac{aE}{5\cdot10^7h}.
\]
After
$t_E=\lceil(3/a)\log(10^7h)\rceil$ steps it follows that
\begin{equation}\label{eq:epochP}
 p_{t_E}\le E/(10^5h).
\end{equation}

Refresh the fast reference at $v'=x_{t_E}$. Run $s_f=O(1)$ batch gradient steps
on $g(\cdot,v')$ from the old $\bar u$, with contraction at most $10^{-6}$ and
batch size
\[
 b_f=\max\left\{1,\left\lceil
             \frac{10^6\sigma^2m}{\alpha\ell E}\right\rceil\right\}.
\]
The starting fast error is at most $3E/m$ in expectation, by
\eqref{eq:epochP} and the same Lipschitz estimate. Thus the refreshed point $u'$
satisfies $m\E\norm{u'-u^*(v')}^2\le E/10^5$.
The triangle inequality in squared norm,~\eqref{eq:epochP}, and $k^2<h$ show
\[
 \E[Q(v')-Q(v^*)+m\norm{u'-u^*}^2]
 \le\frac{7E}{10^5}\le E/2.
\]
The expected epoch cost is at most
\[
 C\mathcal L^2\left[\sqrt{H/m}+\frac{(K+1)\sigma^2}{mE}\right].
\]
The refresh cost is included: $m^2/(\alpha\ell)\le1600$ under
\eqref{eq:valuerange}, so it is no larger than a universal constant times
$1+\sigma^2/(mE)$.

Start with envelope $\max\{E_0,e\}$ and halve it epoch by epoch, ending with
an epoch whose envelope is in $[e,2e]$. The last epoch bounds both required
errors by $e$. There are $O(\log(2+E_0/e))$ epochs, while the sum of their
inverse envelopes is at most $2/e$ (up to an immaterial endpoint factor).
Since $K+1\le C\mathcal L$, summing the costs proves~\eqref{eq:valuecost}.
Every inequality used either holds conditionally on the current query or uses
the stated expected envelope. This also proves the assertion for random initial
data and random functions. In particular, the level cutoff and the noise cost
are independent of $E_0$; there is no hidden initialization logarithm multiplying
$\sigma^2/(me)$.
\end{proof}

\subsection{Recursive regularization and a saddle-gap certificate}
The following elementary identity supplies the gradient-small conversion. We
include its proof instead of importing a solver with a different oracle or domain.

\begin{lemma}[Recursive value-to-gradient conversion]\label{lem:RRnew}
Let $q$ be $H_0$-smooth and $\mu$-strongly convex. Put
$J=1+\lceil\log_2(H_0/\mu)\rceil$, $m_j=2^j\mu$, and $q_0=q$.
For $j=1,\ldots,J$, obtain $v_j$ with
\[
 \E[q_{j-1}(v_j)-\min q_{j-1}]\le e_j,
 \qquad e_j=\frac{\delta^2}{10^4J^2m_{j-1}},
\]
and define $q_j=q_{j-1}+(m_{j-1}/2)\norm{\,\cdot-v_j}^2$.
Then $\E\norm{\nabla q(v_J)}^2\le\delta^2$.
Unconditional expectations suffice, including when the successive functions are random.
\end{lemma}
\begin{proof}
Let $z_j=\arg\min q_j$. Since
$q_j(v_j)-\min q_j\le q_{j-1}(v_j)-\min q_{j-1}$, strong convexity gives
$\E\norm{v_j-z_j}^2\le2e_j/m_j$.
Also $\E\norm{v_{j+1}-z_j}^2\le2e_{j+1}/m_j$, so
\[
 \E\norm{v_{j+1}-v_j}^2\le4(e_j+e_{j+1})/m_j.
\]
The smoothness of $q_{J-1}$ is at most
$H_0+m_{J-1}-\mu\le2m_{J-1}$, giving
$\E\norm{\nabla q_{J-1}(v_J)}^2\le4m_{J-1}e_J$.
The exact telescoping identity is
\[
 \nabla q(v_J)=\nabla q_{J-1}(v_J)
       -\sum_{j=2}^J(m_{j-1}-\mu)(v_j-v_{j-1}).
\]
The triangle inequality in $L^2$ bounds the square root of the desired expected
squared norm by
$2\delta/(100J)+(J-1)\sqrt{12}\delta/(100J)<\delta$.
No conditional accuracy assertion is needed for this calculation.
\end{proof}

\begin{theorem}[Unconditional localized SCSC primitive]\label{thm:innernew}
Let $g$ be jointly $\ell$-smooth, $\alpha$-strongly convex in $u$, and
$\mu$-strongly concave in $v$, with $\alpha\ge\ell/4$. Let $(u^*,v^*)$ be its
saddle, $\kappa_g=\ell/\mu$, and suppose
\[
 \E[\ell\norm{u_0-u^*}^2+\mu\norm{v_0-v^*}^2]\le\mathcal E_0.
\]
There is a fresh-gradient algorithm returning $(u,v)$ with
$\E\Gap_g(u,v)\le\eta$, where
$\Gap_g(u,v)=\max_wg(u,w)-\min_zg(z,v)$, at expected original-call cost
\begin{equation}\label{eq:innercomplete}
 C\Lambda_g^8\left[
 \sqrt{\kappa_g}\log\left(2+\frac{\kappa_g\mathcal E_0}{\eta}\right)
                 +\frac{\sigma^2}{\mu\eta}\right],\qquad
 \Lambda_g=1+\log(2+\kappa_g)+\log\left(2+\frac{\sigma^2}{\mu\eta}\right).
\end{equation}
The same statement holds with random $g$ and random initialization satisfying
these uniform assumptions and expected energy bound. All auxiliary minimizations
are implemented by the explicit fresh-gradient steps above.
\end{theorem}
\begin{proof}
Use $q=-\min_u g$, $H_0=5\ell$, and Lemma~\ref{lem:RRnew} with
$\delta^2=\mu\eta/100$. Implement each value solve by Lemma~\ref{lem:valuesolver}.
The regularizations are known dual quadratics, leaving the fast minimizer unchanged.
For stage $j$, $m=m_{j-1}$, and one can use
$H=H_0+m_{j-1}-\mu\le15\ell$, so~\eqref{eq:valuerange} holds.
The value solver also returns a fast point $u_j$ with
$m_{j-1}\E\norm{u_j-u^*(v_j)}^2\le e_j$.

The first starting value energy is at most $C\kappa_g\mathcal E_0$, by
$q(v_0)-\min q\le(H_0/2)\norm{v_0-v^*}^2$.
For subsequent stages use $(u_j,v_j)$ as the reference for $q_j$.
Its new value gap is at most $e_j$ in expectation. Moreover, since
$u^*(\cdot)$ is $4$-Lipschitz and
$\E\norm{v_j-z_j}^2\le2e_j/m_j$,
\[
 m_j\E\norm{u_j-u^*(z_j)}^2
 \le4e_j+4\cdot16e_j\le68e_j.
\]
Thus $69e_j$ is a deterministic expected starting envelope for the next solve.
The ratio $69e_j/e_{j+1}=138$ is universal: subsequent initialization logarithms
cannot depend on the cold-start energy.

At the final stage $m_{J-1}\ge H_0$, and hence
\[
 \ell^2\E\norm{u_J-u^*(v_J)}^2
 \le\frac{\ell^2\delta^2}{10^4J^2m_{J-1}^2}\le\delta^2/10^4.
\]
Lipschitzness gives
$\E\norm{\nabla_u g(u_J,v_J)}^2\le\delta^2/10^4$ and
$\E\norm{\nabla_v g(u_J,v_J)}^2\le3\delta^2$.
Strong convexity and strong concavity give the pointwise residual certificate
\[
 \Gap_g(u,v)\le\frac{\norm{\nabla_u g(u,v)}^2}{2\alpha}
                     +\frac{\norm{\nabla_v g(u,v)}^2}{2\mu}.
\]
Since $\mu\le\ell\le4\alpha$, the chosen $\delta$ proves the desired gap bound.

For cost accounting, $J=O(\log(2+\kappa_g))$, and at every stage
$\sigma^2/(m_{j-1}e_j)=10^4J^2\sigma^2/\delta^2$.
All multilevel logarithms are bounded by $C\Lambda_g$ and are independent of
$\mathcal E_0$. The sum of $\sqrt{H/m_{j-1}}$ is $O(\sqrt{\kappa_g})$:
use $\sqrt{H_0/m+1}\le\sqrt{H_0/m}+1$ and $J=O(\sqrt{\kappa_g})$.
Only the first stage has an initialization logarithm, bounded by
\[
 C\log(2+\kappa_g)+C\log(2+\kappa_g\mathcal E_0/\eta).
\]
The total statistical term before the multilevel overhead is
$10^4J^3\sigma^2/\delta^2$.
Substituting these bounds into~\eqref{eq:valuecost} proves~\eqref{eq:innercomplete};
the power eight is a deliberately loose universal logarithmic envelope.
When $\sigma=0$ every call uses the deterministic branch of
Lemma~\ref{lem:reloracle}, proving this endpoint directly.
\end{proof}

\section{Outer Algorithm and a Fixed Query Budget}\label{sec:outernew}

\begin{theorem}[Localized unrestricted upper bound]\label{thm:uppernew}
In the nontrivial regime $L\Delta\ge\epsilon^2$, there is a fixed-budget fresh-gradient
algorithm with success probability at least $2/3$ and query bound
\begin{equation}\label{eq:uppercomplete}
 C\Lambda^{10}\left[
 A+B+C+\sqrt\kappa\log\left(2+\frac{\kappa^2G^2}{\epsilon^2}
                                      +\frac{\kappa^2L\Delta}{\epsilon^2}\right)
 \right],
\end{equation}
where the $C$ inside the brackets denotes the statistical term defined in
Section~1, the outside $C$ is a universal constant, and
\[
 \Lambda=1+\log(2+\kappa)+\log(2+L\Delta/\epsilon^2)
                           +\log(2+\kappa^2\sigma^2/\epsilon^2).
\]
In particular the prefactor $\Lambda^{10}$ is independent of $G$.
\end{theorem}
\begin{proof}
We first analyze an algorithm with random, almost surely finite running cost.
Strong concavity gives $\norm{y^*(x)-y^*(x')}\le\kappa\norm{x-x'}$.
Danskin's identity implies that $\Phi$ is $(1+\kappa)L$-smooth, and $\Phi$ is
$L$-weakly convex as a pointwise maximum of $L$-weakly convex functions.
For a center $c$, define
\begin{gather*}
 g_c(u,v)=f(u,v)+L\norm{u-c}^2,\qquad
 P_c(u)=\Phi(u)+L\norm{u-c}^2,\\
 u_c^*=\arg\min P_c,\qquad m(c)=2L(c-u_c^*).
\end{gather*}
The subproblem is $3L$-smooth, $L$-strongly convex in $u$, and $\mu$-strongly
concave in $v$, so Theorem~\ref{thm:innernew} applies. Also
\[
 m(c)=\nabla\Phi(u_c^*),\qquad
 P_c(c)-P_c(u_c^*)\ge\norm{m(c)}^2/(8L).
\]
Set, for definiteness,
\[
 \eta_c=\frac{\epsilon^2}{10^6L},\qquad
 T=\left\lceil\frac{10^6L\Delta}{\epsilon^2}\right\rceil.
\]
For $t=0,\ldots,T-1$, solve $g_{x_t}$ to expected saddle gap $\eta_c$, and use
its output $(x_{t+1},y_{t+1})$ as the next reference. The gap dominates primal
suboptimality, so unconditionally
\[
 \E\Phi(x_{t+1})+L\E\norm{x_{t+1}-x_t}^2
 \le\E\Phi(x_t)-\frac{\E\norm{m(x_t)}^2}{8L}+\eta_c.
\]
Telescoping yields
\begin{equation}\label{eq:outerexpected}
 \frac1T\sum_{t<T}\E\norm{m(x_t)}^2
 \le8L(\Delta/T+\eta_c)\le16\epsilon^2/10^6,
 \quad L\sum_{t<T}\E\norm{x_{t+1}-x_t}^2\le\Delta+T\eta_c.
\end{equation}

We justify the expected energy envelopes used by the inner solver; no unobserved
radius is passed to the algorithm. For the first subproblem,
$\norm{u_0^*}^2\le2\Delta/L$ and
$\norm{y^*(0)}\le G/\mu$. Consequently its energy from the origin is at most
$C(G^2/\mu+\kappa\Delta)$, a known parameter bound.
For two centers, strong monotonicity of the saddle operator gives
\[
 L\norm{u_c^*-u_{c'}^*}^2+\mu\norm{v_c^*-v_{c'}^*}^2
 \le2L\langle c-c',u_c^*-u_{c'}^*\rangle
 \le\frac L2\norm{u_c^*-u_{c'}^*}^2+2L\norm{c-c'}^2.
\]
A saddle gap at most $\eta_c$ in expectation bounds the preceding weighted
squared error by $2\eta_c$. Together with~\eqref{eq:outerexpected}, a common
known envelope $C(\eta_c+\Delta+T\eta_c)$ therefore bounds every subsequent
expected starting energy. This establishes the inductive hypotheses for
Theorem~\ref{thm:innernew}, even though each subproblem is random.
All its warm-start logarithms are independent of $G$.

Choose $R$ uniformly from $\{0,\ldots,T-1\}$, independently of the coarse solves.
Starting from its stored output $(x_{R+1},y_{R+1})$, refine the \emph{same}
subproblem $g_{x_R}$ to expected gap
\[
 \eta_f=\frac{\epsilon^2}{10^6\kappa L},
\]
and let $\widehat x$ be its primal output. Its expected starting energy is
$O(\eta_c)$. The function $P_{x_R}$ is at most $4\kappa L$-smooth, so
\[
 \E\norm{\nabla P_{x_R}(\widehat x)}^2\le8\kappa L\eta_f,
 \qquad L\norm{\widehat x-u_{x_R}^*}\le\norm{\nabla P_{x_R}(\widehat x)}.
\]
Using the exact identity
\[
 \nabla\Phi(\widehat x)=m(x_R)+\nabla P_{x_R}(\widehat x)
                              -2L(\widehat x-u_{x_R}^*),
\]
we obtain
\[
 \E\norm{\nabla\Phi(\widehat x)}^2
 \le2\E\norm{m(x_R)}^2+18\E\norm{\nabla P_{x_R}(\widehat x)}^2
 \le176\epsilon^2/10^6<\epsilon^2/12.
\]

The coarse statistical cost is $\widetilde O(T\sigma^2/(\mu\eta_c))=
\widetilde O(B)$; the fine statistical cost is
$\widetilde O(\sigma^2/(\mu\eta_f))=\widetilde O(C)$.
The coarse deterministic cost is $\widetilde O(A)$ plus the single cold-start
logarithm displayed in~\eqref{eq:uppercomplete}; the fine deterministic cost is
absorbed by $A$ and its permitted logarithms. The explicit envelope
$\Lambda^{10}$ dominates all logarithms in~\eqref{eq:innercomplete}, including
warm initialization ratios and the original parameters, except the one displayed
$G$-dependent logarithm. Thus the right side of~\eqref{eq:uppercomplete}, with
an appropriate universal constant, bounds the expected \emph{total original-call}
cost by a known number $\mathcal B$.

Finally execute this algorithm with a hard limit $\lceil12\mathcal B\rceil$
on original oracle calls. If it would exceed that limit, terminate and output
zero. Couple it with the preceding untruncated algorithm using the same random
seeds. Markov's inequality gives probability at most $1/12$ of truncation.
The probability of an unsuccessful untruncated output is at most $1/12$, also
by Markov's inequality applied to the last squared-gradient bound. Therefore the
fixed-budget algorithm succeeds with probability at least $5/6$, which is more
than required. Absorbing the factor twelve proves~\eqref{eq:uppercomplete}.
Every cutoff concerns original calls, including the random multilevel costs.
\end{proof}

\section{Proof of the Main Results}\label{sec:closure}
Theorems~\ref{thm:Anew},~\ref{thm:B},~\ref{thm:C} and
Corollary~\ref{cor:D} give four admissible hard distributions, not a direct sum
requiring simultaneous budgets. The maximum of their lower bounds is at least
a universal fraction of their sum. Thus, after increasing the universal
nontrivial-regime constants if needed,
\[
 \Tloc\ge c(A+B+C+D_G).
\]
Together with Theorem~\ref{thm:uppernew} this proves the main theorem.
To see the stated logarithmic equivalence explicitly, write
$b=G/(\epsilon\sqrt\kappa)$. The cold-start logarithm is bounded by
\[
 C\big[1+\log(2+\kappa)+\log(2+L\Delta/\epsilon^2)+\log_+b\big].
\]
Since $A\ge\sqrt\kappa$ in the nontrivial regime, the first three contributions
are absorbed into the permitted polylogarithmic multiple of $A$; the last is
exactly a universal multiple of $D_G$. The prefactor does not conceal any
additional dependence on $G$.

The deterministic endpoint uses the exact-oracle lower constructions in
Sections~\ref{sec:D} and~\ref{sec:Anew} and the deterministic branch of the
upper solver. Therefore
\[
 \Tloc(L,\mu,\Delta,0,G,\epsilon)=\widetilde\Theta(A+D_G)
\]
holds directly, not by passing to a limit in a noise-dependent construction.
Removing $G$ gives Theorem~\ref{thm:paper-infinite}, already proved in
Section~\ref{sec:D}. This completes the localized characterization and the unlocalized impossibility result.

These assertions keep the nontrivial regime in the problem specification.
They do not assert the same positive lower formula at $\Delta=0$: then the
origin is already a global primal minimizer. Nor do they assert a positive
$G$-dependent lower bound at $\kappa=1$, where joint $L$-smoothness and
$L$-strong concavity force the primal--dual cross derivative to vanish.
Dimensions in the lower bounds may be very large but are explicit finite
integers in~\eqref{eq:Ddimension},~\eqref{eq:Bdimensions}, and~\eqref{eq:Adim}.
Every hard oracle is fixed before interaction, and every upper query uses only
one fresh joint-gradient sample at one point. None of the conclusions presumes
that a realized noisy vector is the gradient of a smooth sample function.

\end{document}